\documentclass{article}[11]
\usepackage{amsmath, amsthm, amssymb, color, amsrefs, authblk,arydshln}
\usepackage{mathrsfs,mathtools,enumerate,hyperref}

\newtheorem{thm}{Theorem}[section]
\newtheorem*{thm*}{Theorem}
\newtheorem{prop}[thm]{Proposition}
\newtheorem{lem}[thm]{Lemma}
\newtheorem{cor}[thm]{Corollary}

\theoremstyle{definition}
\newtheorem{definition}[thm]{Definition}
\newtheorem{example}[thm]{Example}
\newtheorem*{ex*}{Example}

\theoremstyle{remark}
\newtheorem{remark}[thm]{Remark}
\newtheorem*{remark*}{Remark}

\DeclareMathOperator\vol{vol}

\DeclareMathOperator\SO{SO}
\DeclareMathOperator\U{U}

\DeclareMathOperator\Hess{Hess}
\DeclareMathOperator\ver{ver}
\DeclareMathOperator\hor{hor}

\DeclareMathOperator\Ric{Ric}

\DeclareMathOperator\Vol{Vol}

\def\div{\mbox{div}}
\def\End{\mbox{End}}
\def\i{\mathrm{i}}

\def\w{\wedge}
\DeclareMathOperator\tr{tr}
\def\C{\mathbb{C}}

\def\R{\mathbb{R}}

\def\Z{\mathbb{Z}}

\def\V{\mathcal{V}}

\begin{document}

\title{Variation formulae for the volume\\of coassociative submanifolds}

\author{Tommaso Pacini and Alberto Raffero}
\affil{Department of Mathematics, University of Torino \\ via Carlo Alberto 10, 10123 Torino, Italy \\ tommaso.pacini@unito.it, alberto.raffero@unito.it}

\maketitle

\begin{abstract}
We prove new variation formulae for the volume of coassociative submanifolds, expressed in terms of $G_2$ data. 
These formulae highlight the role of the ambient torsion and Ricci curvature. 
As a special case, we obtain a second variation formula for variations within the moduli space of coassociative submanifolds.  
These results apply, for example, to coassociative fibrations. 
We illustrate our formulae with several examples, both homogeneous and non.
\end{abstract}

\section{Introduction}\label{s:intro}
This paper concerns the calculus of variations of the volume functional for submanifolds, in a specific geometric context. 
It is based upon two guiding principles.
\paragraph{Variation formulae in calibrated geometry.} 
Recall, in the general setting of Riemannian geometry, the first variation formula for volume (notation explained below):
$$\frac{d}{dt}\Vol(\Sigma_t)_{|t=0}=-\int_\Sigma H\cdot Z^\perp\vol_0.$$
It is an interesting question whether this formula, and/or the second variation formula, 
have alternative expressions in the more restrictive context of calibrated geometry \cite{HL}.

Calibrated geometry is defined by a global differential form $\alpha$ such that $|\alpha|\leq 1$. We then say that a submanifold $\Sigma$ is calibrated if 
$|\alpha_{|\Sigma}|=1$; equivalently, $\alpha_{|\Sigma}=\vol_{\Sigma}$. In this context there are two fundamental facts that should be taken into account:
\begin{enumerate}[$\bullet$]
\item $\alpha$ provides an extension of the volume form of a calibrated submanifold to a global form on the ambient space. 
One might expect that, in this context, an alternative first variation formula will incorporate ambient data related to the calibration, 
rather than data defined only along the submanifold by the Riemannian structure (such as $H$). 
\item If one adds the condition $d\alpha=0$ then any calibrated submanifold is automatically area-minimizing, 
thus in particular minimal. Any variation formula should reproduce this fact. 
\end{enumerate}
Analogous considerations hold for the second variation formula. In particular, the area-minimizing property implies stability, i.e., 
the second variation is non-negative. 
In this sense, in the closed case $d\alpha=0$ most of the information typically contained in the variation formulae is already known by other means. 

We shall be interested in the more general scenario, in which closedness is not required. 

\paragraph{$\mathbf{G_2}$ vs.~K\"ahler geometry.} 
More specifically, the goal of this paper is to show how the above plays out in the context of $G_2$ geometry. 
Here, one deals with 7-dimensional manifolds endowed with a certain non-degenerate 3-form $\phi$, 
which we shall always assume to be closed: $d\phi=0$. 
From some points of view, this set-up strongly resembles the case of $2n$-dimensional K\"ahler manifolds, which are endowed with a non-degenerate closed 2-form $\omega$. 
Indeed, the formal analogies between these two geometries have classically been one of the guidelines in the development of $G_2$ geometry. 
Concerning submanifolds $\Sigma$, the analogy continues on the $G_2$ side with the class of coassociative submanifolds, 
defined by the condition $\phi_{|\Sigma}=0$, on the K\"ahler side with the class of Lagrangian submanifolds, 
defined by the condition $\omega_{|\Sigma}=0$. 

The link to the first point, above, is the fact that coassociative submanifolds can alternatively be described as calibrated by $\psi\coloneqq\star\phi$, 
the Hodge dual 4-form. 

Putting all this together, the main goal of this paper is to investigate the first and second variation formulae for coassociative submanifolds, 
in the general scenario in which $\psi$ is not closed. In deference to the above analogy, 
along the way we shall compare our results to a second variation formula obtained by Oh \cite{Oh} for Lagrangian submanifolds. 

The analogy between $G_2$ and K\"ahler geometry breaks down however on two vital, inter-related, issues. 
\begin{enumerate}[1.]
\item In the $G_2$ case all other structures (metric, orientation, etc.) are generated by $\phi$ itself. 
This is in stark contrast with the K\"ahler condition, which can be defined as a compatibility condition between three a priori independent structures 
$(J,g,\omega)$: the remaining degrees of freedom, i.e., the possibility of ``playing one structure against the other", 
provides a useful handle in many directions of research in K\"ahler geometry.  
\item K\"ahler manifolds are automatically torsion-free, in the sense that the Levi-Civita connection defined using the $\SO(2n)$-frame bundle restricts 
to the $\U(n)$-frame bundle. $G_2$ manifolds are torsion-free in the analogous sense if and only if a second condition is imposed: $d\psi=0$. 
We have already mentioned that this condition has strong implications for coassociative submanifolds.  
More generally, it has very strong implications for the whole of $G_2$ geometry, 
suddenly pushing it much closer to the geometry of Calabi-Yau manifolds: a very small part of the wide world of K\"ahler geometry. 
For both reasons, we will avoid assuming $d\psi=0$.
\end{enumerate}

Our main result is a new second variation formula, specific to coassociative submanifolds: Theorem \ref{thm:secvar}. 
One of the typical goals of such a formula is to allow us to detect geometric conditions ensuring that a minimal submanifold is also stable. 
This is achieved in Corollary \ref{cor:stability}.

Along the way, in order to facilitate comparisons, we discuss both the standard and the Lagrangian variation formulae. 
In this regard we wish to emphasize the following points:
\begin{enumerate}[1.]
\item The standard second variation formula contains a non-geometric jumble of curvature terms. 
One of the main merits of Oh's Lagrangian formula is to show how, in the K\"ahler context, these terms can be rearranged into the Ricci curvature. 
Interestingly, our formula for coassociative submanifolds also contains a Ricci term.
\item As perhaps should be expected, torsion forces itself into the second variation formula for coassociative submanifolds leading to a new challenge: 
in order to ensure stability of the submanifold, it is also necessary to control the sign of the torsion term. 
In Section \ref{secTorsion} we show how this can be done in the context of $G_2$ manifolds satisfying a natural ``quadratic condition", 
introduced by Bryant \cite{Bryant}. This is the key ingredient to Corollary \ref{cor:stability}.  
\end{enumerate}

The last three sections present several classes of examples, which also serve as a useful testing ground for our results. 
Section \ref{s:erp} shows how our formulae play out in the context of Bryant's ``Extremally Ricci-Pinched" $G_2$ manifolds. 
Here, we study two concrete ambient manifolds due to Bryant and to Lauret, and several families of coassociative submanifolds, 
some of which are new in the literature. Section \ref{s:example} provides a reformulation of our formulae in the context of coassociative fibrations 
defined by a Riemannian submersion. 
This is applied to a classical example due to Fern\'andez \cite{Fernandez}. Section \ref{s:perturbations} presents a method to build new, 
non-homogeneous, examples by perturbing the Fern\'andez closed $G_2$ structure.

\ 

\noindent\textbf{Acknowledgements.} The first variation formula for coassociative submanifolds presented below may be already known. 
Jason Lotay and TP came across it, in joint unpublished work, many years ago. 
Second variation formulae for coassociative submanifolds in the case $d\psi=0$ have been obtained in \cite{McLean} and \cite{LeVanzura}.
The authors were supported by the project PRIN 2022  
``Real and Complex Manifolds: Geometry and Holomorphic Dynamics'' and by GNSAGA of INdAM. 
The authors wish to thank Anna Fino, Luciano Mari and Fabio Podest\`a for interesting conversations.

\section{First variation formulae}\label{s:first}
Let us start by reviewing the classical formula.

\paragraph{The classical setting.}
Let $\Sigma^k$ be a compact oriented manifold. Let $(M^n,g)$ be a Riemannian manifold. 
Let $\iota_t:\Sigma\hookrightarrow M$ be a curve of immersions. The induced metric defines volume forms $\vol_t$ on the image submanifolds 
$\Sigma_t$. We may pull them back to $\Sigma$, obtaining the curve of volume forms $\iota_t^*\vol_t$. 
This allows us to compute the volume of the image submanifolds using the fixed manifold $\Sigma$: $\Vol(\Sigma_t)=\int_\Sigma \iota_t^*\vol_t$.

We are interested in understanding how the volume changes with $t$. As a first step, let us investigate the variation of the volume forms. 

Since the bundle of volume forms is trivial there exists a family of functions $f(t):\Sigma\rightarrow\R$ such that
$$\iota_t^*\vol_t=f(t)\,\iota_0^*\vol_0.$$ 
The variation of $\iota_t^*\vol_t$ corresponds to the $t$-derivative of $f(t)$.

Concerning notation, set $Z\coloneqq\frac{d}{dt}\iota_t$. With respect to the submanifolds, 
we can write $Z$ as a sum of tangential and normal components: $Z=Z^\top+Z^\perp$. 
Let $H=\tr_\Sigma(\nabla^\perp)$ denote the mean curvature vector field of $\Sigma_0$. 

\begin{lem}\label{l:firstvar_classical}
$\frac{d}{dt}f(t)_{|t=0}=\div_{\Sigma}(Z^\top)-H\cdot Z^\perp$.
\end{lem}
\begin{proof}
Since both sides are well defined independently of coordinate systems, to prove the statement at any point $p\in\Sigma$ 
we can choose normal coordinates with respect to the metric at time $t=0$. This provides a local basis $v_1,\dots,v_k$ which is orthonormal at $p$. 
The equality $\iota_t^*\vol_t=f(t)\iota_0^*\vol_0$ then implies that, at $p$, $f(t)=|\iota_{t*}v_1\wedge\dots\wedge\iota_{t*}v_k|.$

Notice that
\begin{align*}
\frac{d}{dt}f(t)_{|t=0}&=\left.\frac12\frac{2(\nabla_Z(\iota_{t*}v_1\wedge\dots\wedge\iota_{t*}v_k),\iota_{t*}v_1\wedge\dots\wedge\iota_{t*}v_k)}{|\iota_{t*}v_1\wedge\dots\wedge\iota_{t*}v_k|}\right|_{t=0}\\
&=\sum_{i=1}^k
(\iota_{t*}v_1\wedge\dots\wedge\nabla_Z(\iota_{t*}v_i)\wedge\dots\wedge\iota_{t*}v_k,\iota_{t*}v_1\wedge\dots\wedge\iota_{t*}v_k)_{|t=0}.
\end{align*}
By construction $[Z,\iota_{t*}v_i]=0$ so $\nabla_Z(\iota_{t*}v_i)=\nabla_{\iota_{t*}v_i}Z$. Substituting $t=0$ we find
\begin{equation*}
\frac{d}{dt}f(t)_{|t=0}=\sum_{i=1}^k(v_1\wedge\dots\wedge\nabla_{v_i}Z\wedge\dots\wedge v_k,v_1\wedge\dots\wedge v_k).
\end{equation*}
Clearly, the only relevant component of $\nabla_{v_i}Z$ is $(\nabla_{v_i}Z\cdot v_i)v_i$. 
Separating the tangential and perpendicular components and using orthogonality we find
\begin{align*}
\frac{d}{dt}f(t)_{|t=0}&=(\nabla_{v_i}Z^\top\cdot v_i)|v_1\wedge\dots\wedge v_k|^2+(\nabla_{v_i}Z^\perp\cdot v_i)|v_1\wedge\dots\wedge v_k|^2\\
&=\div_\Sigma(Z^\top)-H\cdot Z^\perp.
\end{align*}
\end{proof}
\begin{cor}
$\frac{d}{dt}\Vol(\Sigma_t)_{|t=0}=-\int_\Sigma H\cdot Z^\perp\iota_0^*vol_0$.
\end{cor}
In other words, $-H$ is the $L^2$-gradient of $\Vol$, defined on the infinite-dimensional space of submanifolds in $M$.

\paragraph{$\mathbf{G_2}$ manifolds and coassociative submanifolds.}
Let $M^7$ be a $7$-manifold endowed with a $G_2$ structure $\phi$ (not necessarily closed or co-closed).  
Let $\psi=\star\phi$, where $\star$ denotes the Hodge operator determined by the metric $g_\phi$ and the orientation induced by $\phi$. 
We shall follow the conventions in Bryant \cite{Bryant}, so the pointwise models on $\R^7$ for $\phi$ and $\psi$ are
\begin{align*}
\phi&=123+1(45+67)+2(46-57)-3(47+56),\\
\psi&=4567+23(45+67)-13(46-57)-12(47+56),
\end{align*}
where 123 (or, at times, $e^{123}$) is short-hand for $e^1\wedge e^2\wedge e^3$, etc. 
Notice, in both cases, the appearance of self-dual forms with respect to 4,5,6,7. 
Moreover, we shall use Bryant's compact notation 
\begin{equation*}
\phi	=	\frac{1}{6}\sum_{i,j,k}\epsilon_{ijk}\,e^i\wedge e^j\wedge e^k=\sum_{i<j<k}\epsilon_{ijk}\,e^i\wedge e^j\wedge e^k
	=	\sum_{i,j,k}\epsilon_{ijk}\,e^i\otimes e^j\otimes e^k, 
\end{equation*}
where the $\epsilon_{ijk}\in\{\pm 1,0\}$ are anti-symmetric with respect to the indices and are chosen so as to reproduce the initial expression for $\phi$. 

\smallskip

Both $\phi$ and $\psi$ satisfy the calibration inequality:
\begin{enumerate}[$\bullet$]
\item for any oriented 3-plane $\pi$ in $TM$, $|\phi_{|\pi}|\leq\vol_\pi$. The 3-plane is called associative if equality holds; 
\item for any oriented 4-plane $\pi$ in $TM$, $|\psi_{|\pi}|\leq\vol_\pi$. The 4-plane is called coassociative if equality holds. 
\end{enumerate}
It is important to stress the fact that associative and coassociative planes have a canonical orientation: that for which $\vol_\pi=\phi_{|\pi}$. 
We shall always consider them with this orientation.

A 3-plane is associative if and only if its normal 4-plane is coassociative. 
The group $G_2$ acts transitively on the space of associative 3-planes, thus also on the space of coassociative 4-planes 
(\cite[Theorem 1.8, p.~114]{HL}).  
We may thus always assume that, in the pointwise model, a coassociative plane $\pi$ is spanned by $e_4,e_5,e_6,e_7$ and that its normal space is 
spanned by $e_1,e_2,e_3$. Furthermore, the isotropy subgroup of $\pi$ acts like $\SO(3)$ on the normal space (\cite[eq.~(1.9) p.~115]{HL}). 
We may therefore also assume that any given normal vector $Z$ coincides with $e_1$. This often facilitates studying the behaviour of certain tensors and 
contractions in terms of the pointwise models of $\psi$ and $\phi$. In particular, the operation $Z\mapsto(Z\lrcorner\phi)_{|\pi}$ produces an isomorphism 
between the normal (associative) space $\pi^\perp$ and the space of self-dual 2-forms $\Lambda^2_+$ on $\pi$.

Concerning the above inequalities, Harvey-Lawson \cite{HL} provide explicit characterizations of the difference. Regarding $\psi$, they show: $\forall v_1,\dots,v_4\in T_pM$, 
\begin{equation*}(\psi(v_1,\dots,v_4))^2+|\mathcal{C}(v_1,\dots,v_4)|^2=|v_1\wedge\dots\wedge v_4|^2, 
\end{equation*}
where $\mathcal{C}\in\Gamma(T^*M^{\otimes4}\otimes TM)$ is the coassociator tensor, depending linearly on $\phi$.

Let $\Sigma$ be a 4-dimensional compact oriented manifold. We will say that an immersion $\iota:\Sigma\hookrightarrow M$ is coassociative if $\iota^*\psi=\iota^*\vol$; equivalently, up to orientation, $\iota^*\phi=0$. We will often identify $\Sigma$ with its image. 

As mentioned in the Introduction, the fundamental lemma of calibrated geometry shows that, if $d\psi=0$, then any coassociative submanifold is automatically volume-minimizing in its homology class. In particular it is minimal, thus $H=0$, and stable. We shall be interested in understanding what happens in the alternative setting $d\phi=0$.

Recall, cf.~e.g.~\cite{G2lectures}, that there exists a global endomorphism $T$ on $M$ such that $\nabla_Z\phi=T(Z)\lrcorner\psi$. 
Recall also the decomposition of the bundle of $p$-forms on $M$ into $G_2$-irreducible subspaces $\Lambda^p_k$, where $k$ denotes the dimension of the subspace. In particular,
\begin{equation*}
\Lambda^1=\Lambda^1_7,\ \ \Lambda^2=\Lambda^2_7\oplus\Lambda^2_{14}, \ \ \Lambda^3=\Lambda^3_1\oplus\Lambda^3_7\oplus\Lambda^3_{27}.
\end{equation*}
Up to the standard musical isomorphisms, one can apply this decomposition to $T$. It turns out that, when $\phi$ is closed, $T(Z)=\tau_2(Z,\cdot)^\sharp$, where $\tau_2\in\Lambda^2_{14}$. One also finds $d\psi=\tau_2\wedge\phi = -\star \tau_2$. 

The tensor $T$ (or $\tau_2$) is known as the torsion of $\phi$: it is apparent from the above formulae that it appears naturally in calculations. In particular, one can prove that $T=0$ if and only if $\psi$ is also closed; 
equivalently, if and only if $\phi$ and $\psi$ are parallel with respect to the Levi-Civita connection.

\paragraph{The first variation formula for coassociatives.}
Let $\iota_t:\Sigma\hookrightarrow M$ be a curve of immersions as above, such that $\iota_0$ is coassociative.

\begin{lem}\label{l:firstvar_G2}
Assume $d\phi=0$. Then
$$\frac{d}{dt}f(t)_{|t=0}\,\iota_0^*\vol_0=(d\iota_{Z^\top}\psi+\tau_2\wedge(\iota_{Z^\perp}\phi))_{|\Sigma}.$$
\end{lem}
\begin{proof}
Using normal coordinates, identifications and notation as above, 
\begin{align*}
\frac{d}{dt}f(t)_{|t=0}&=\frac{d}{dt}\sqrt{\psi(\iota_{t*}v_1,\dots,\iota_{t*}v_4)^2+|\mathcal{C}(\iota_{t*}v_1,\dots,\iota_{t*}v_4)|^2}_{|t=0}\\
&=(1/2)\frac{2\psi\dot\psi+2(\mathcal{C}\cdot\mathcal{C}_Z)}{\sqrt{\psi^2+|\mathcal{C}|^2}}_{|t=0}\\
&=\dot\psi_{|t=0},
\end{align*}
because coassociativity implies $\psi(v_1,\dots,v_4)=1$, $\mathcal{C}(v_1,\dots,v_4)=0$. 
We remark that, above, $\mathcal{C}_Z$ is shorthand for $\nabla_Z(\mathcal{C}(\iota_{t^*}v_1,\dots,\iota_{t^*}v_4))_{|t=0}$.

We now write
\begin{align*}
\dot\psi_{|t=0}&=\frac{d}{dt}(\iota_t^*\psi)_{|t=0}(v_1,\dots,v_4)\\
&=(\mathcal{L}_Z\psi)(v_1,\dots,v_4)\\
&=(d\iota_Z\psi+\iota_Z(\tau_2\wedge\phi))(v_1,\dots,v_4)\\
&=(d\iota_Z\psi+\tau_2\wedge(\iota_Z\phi))(v_1,\dots,v_4),
\end{align*}
because coassociativity implies $\phi_{|\Sigma}=0$. To conclude, coassociativity and the pointwise model show that $\iota_{Z^\perp}\psi\equiv 0$ and $\iota_{Z^\top}\phi\equiv 0$ along $\Sigma$.
\end{proof}
Comparing the formulae in Lemma \ref{l:firstvar_classical} and Lemma \ref{l:firstvar_G2} first for vector fields tangent to $\Sigma$, i.e., $Z=Z^\top$, 
then for normal vector fields, shows that
\begin{align*}
\div_\Sigma(Z^\top)\,\iota_0^*\vol_0&=d\iota_{Z^\top}\psi_{|\Sigma},\\
-(H\cdot Z^\perp)\,\iota_0^*\vol_0&=\tau_2\wedge(\iota_{Z^\perp}\phi)_{|\Sigma}=\tau_2\wedge\star(\iota_{Z^\perp}\phi)_{|\Sigma}=(\tau_2^+\cdot\iota_{Z^\perp}\phi_{|\Sigma})\,\iota_0^*\vol_0. 
\end{align*}
It follows that these equalities hold for any vector field $Z$. 

The further identity $-H\cdot Z^\perp=-\tfrac12\iota_H\phi\cdot\iota_{Z^\perp}\phi_{|\Sigma}$ allows us to identify $-H$ with the self-dual component of the restriction of $\tau_2$. Specifically,
\[
-\iota_{H}\phi_{|\Sigma} = 2\,\tau_{2|\Sigma}^+. 
\] 

In particular, $H=0$ if and only if $\tau_{2|\Sigma}\in\Lambda^2_-(\Sigma)$. 

\begin{cor}Assume $d\phi=0$. Then
$$\frac{d}{dt}\Vol(\Sigma_t)_{|t=0}=\int_{\Sigma}\tau_{2|\Sigma}^+\wedge \iota_{Z^\perp}\phi.$$
\end{cor}

Notice that if also $d\psi=0$ then $\tau_2=0$, so our formulae agree with calibration theory: $\Sigma_0$ is minimal.

\paragraph{Moduli spaces.}
The condition $d\phi=0$ implies that the space of coassociative deformations of an initial compact coassociative submanifold $\Sigma$ 
forms a smooth moduli space $\mathcal{M}$, cf.~e.g.~\cite{G2lectures}. 
The theory shows that 
\begin{enumerate}[(i)] 
\item the isomorphism $T\Sigma^\perp\simeq\Lambda^2_+(\Sigma)$ relates infinitesimal deformations in $\mathcal{M}$ 
to self-dual harmonic 2-forms; 
\item these integrate to actual deformations, so the dimension of $\mathcal{M}$ 
coincides with that of the space of self-dual harmonic 2-forms on $\Sigma$. Hodge theory implies that this dimension is topological: $\dim(\mathcal{M})=b_2^+(\Sigma)$.
\end{enumerate}
The restriction to normal vector fields corresponds to the fact that $\mathcal{M}$ is defined modulo reparametrizations, 
i.e., it contains non-parametrized submanifolds.

Given infinitesimal deformations $Z_1,Z_2\in T_\Sigma\mathcal{M}$, the $L^2$ metric $\int_\Sigma Z_1\cdot Z_2\vol_\Sigma$ defines a canonical Riemannian structure on $\mathcal{M}$.

We could decide to vary $\Sigma$ only within this moduli space. In this case the proof of the first variation formula simplifies, because
$$\frac{d}{dt}\Vol(\Sigma_t)_{|t=0}=\frac{d}{dt}\int_\Sigma\iota_t^*\psi_{|t=0}=\int_\Sigma\mathcal{L}_Z\psi.$$
The proof now continues as before. 
The conclusion is of course weaker: this proof allows us to identify only the $L^2$-projections of $-H$ and $\tau_{2|\Sigma}^+$ onto $T\mathcal{M}$.

\paragraph{Fibrations.}
Assume $M$ (as usual satisfying $d\phi=0$) admits a coassociative fibration over a base space $B$, with compact fibres. 
Let $\Sigma$ be the fibre above $b\in B$. It is natural to view $B$ as a submanifold in the moduli space $\mathcal{M}$ defined by $\Sigma$. 
The relationship between $T_bB$ and $T_\Sigma\mathcal{M}$ is defined by the fact that any $Z\in T_bB$ admits a unique lift to a vector field normal to 
$\Sigma$ (with respect to the metric on $M$). 
The fibration structure implies that this vector field corresponds to a deformation through coassociative fibres, so it lies in $T_\Sigma\mathcal{M}$. 
We can endow $B$ with the metric induced by the metric on $\mathcal{M}$. In particular, $B$ can be locally identified with $\mathcal{M}$ 
if and only if they have the same dimension, i.e., if and only if $b_2^+(\Sigma)=3$: this is a topological condition on $\Sigma$.

The fact that $\Sigma$ fits into a coassociative fibration has further consequences. Indeed, the projection provides an isomorphism between each 
$T_p\Sigma^\perp$ and $T_bB$ so the normal bundle, thus $\Lambda^2_+(\Sigma)$, is trivial. Choose a basis $Z_1,Z_2,Z_3$ for $T_bB$. 
We will use the same notation to denote the corresponding normal vector fields along $\Sigma$. 
The forms $Z_i\lrcorner\phi$ provide a basis for the self-dual forms at each point. Any other self-dual form $\alpha$ on $\Sigma$ must be, at each point, 
a linear combination of these: $\alpha=a_iZ_i\lrcorner\phi$. Coassociative deformation theory implies that  $Z_i\lrcorner\phi$ are harmonic. 

In general, there will be no particular constraint on the pointwise lengths of the lifted vector fields. These lengths are constant if and only if the fibration is 
a Riemannian submersion; this condition leads to further constraints on the fibres. 
Specifically, Baraglia \cite{Baraglia} shows that $B$ can be endowed with a Riemannian structure such that the projection is a Riemannian submersion 
if and only if the fibres, with the induced metric, have a hyper-K\"ahler structure (his proof only requires $d\phi=0$).

In particular, general theory shows that such fibres must be either K3 surfaces or flat tori. In both cases $b_2^+(\Sigma)=3$, so $B\simeq \mathcal{M}$. Baraglia shows that (up to a constant factor) this identification further equates his metric on $B$ with the $L^2$ metric on $\mathcal{M}$. In other words, in this setting the moduli space construction is perfectly aligned to the fibration structure.

In Sections \ref{s:erp}, \ref{s:example} we will find examples of coassociative fibrations whose fibres are flat tori. We may then apply the above.

\begin{remark*}
It should be noted that Baraglia \cite{Baraglia} proved that, for topological reasons, 
if $M$ is compact and both its $G_2$ calibrations are closed (i.e., $d\phi=0$ and $d\psi=0$) then it does not admit smooth coassociative fibrations. 
Our methods and statements however are local, we will never require $M$ compact, nor will we require both calibrations to be closed.

Other references for coassociative fibrations include \cite{GYZ}, which first triggered interest in this topic, and \cite{Donaldson}.
\end{remark*} 

\section{Second variation formulae}\label{s:second}
Let us again start by reviewing the classical formula \cite{Simons}.

\paragraph{The classical setting.}
We shall use the same notation as before, but to simplify we shall assume $Z$ is always perpendicular to $\Sigma_t$: this will not change the volumes. 
The same methods show that, for any $t$, 
$$\frac{d}{dt}\Vol(\Sigma_t)=-\int_\Sigma H\cdot Z\,\iota_t^*\vol_t.$$
The starting point for the second variation formula is the assumption that $\Sigma_0$ is minimal, i.e., $H=0$. 
The main goal is to then detect conditions ensuring that $\Sigma_0$ is a local minimum: stability. 

Differentiating the above expression, we find
\begin{align*}
\frac{d^2}{dt^2}\Vol(\Sigma_t)_{|t=0}&=-\int_\Sigma \nabla_Z(H\cdot Z)\,\iota_0^*\vol_0+\int_\Sigma (H\cdot Z)^2\iota_0^*\vol_0\\
&=-\int_\Sigma \nabla_ZH\cdot Z\,\iota_0^*\vol_0.
\end{align*}

These formulae indicate clearly that, as expected, if $\iota_t$ is a curve of minimal immersions then the volume remains constant 
and $Z$ is a Jacobi vector field. They provide however no means of controlling the stability. 
This can be achieved by examining the integrand more closely, as follows.

\begin{prop}\label{2Classic} 
Assume $Z$ normal and $\Sigma_0$ minimal. Then
$$\frac{d^2}{dt^2}\Vol(\Sigma_t)_{|t=0}=\int_\Sigma(-(\nabla_{e_i}Z\cdot e_j)^2-R(e_i,Z)Z\cdot e_i+(\nabla_{e_i}Z\cdot f_j)^2)\,\iota_0^*\vol_0,$$
where $e_1,\dots,e_k$ is an orthonormal basis of $T_p\Sigma$ at any given point, $f_1,\dots,f_{n-k}$ is an orthonormal basis of $T_p\Sigma^\perp$ and $R$ is the curvature tensor of $M$.
\end{prop}
Basically, the first term is the norm squared of $(\nabla Z)^\top$ (restricted to $\Sigma$), the second term is the trace along $\Sigma$ of the appropriate curvature tensor and the third term is the norm squared of $(\nabla Z)^\perp$ (restricted to $\Sigma$). This explains why the expression is independent of the chosen bases.

\ 

By emphasizing geometrically meaningful components in this way we obtain useful conclusions, such as the following.

\begin{cor}
If $\Sigma_0$ is totally geodesic and the ambient curvature is negative, then $\Sigma_0$ is strictly stable, i.e., $\frac{d^2}{dt^2}\Vol(\Sigma_t)_{|t=0}> 0$.
\end{cor} 

\begin{proof}[Proof of Proposition $\ref{2Classic}$] 
The proof of the proposition is the combination of the following two calculations. As before, $v_1,\dots,v_k$ will denote the local basis defined by normal coordinates on $\Sigma$ at $p$, at time $t=0$. We can extend it in the direction $Z$ using $\iota_{t*}$. 

\begin{enumerate}[1.]
\item  At $p$,
\begin{align*}
-\nabla_Z(H\cdot Z)&=-\nabla_Z (g^{ij}\nabla_{v_i}v_j\cdot Z)\\
&=-(\nabla_Z g^{ij})\nabla_{v_i}v_j\cdot Z-g^{ij}\nabla_Z(\nabla_{v_i}v_j\cdot Z),
\end{align*}
where $g_{ij}=v_i\cdot v_j$ generates a matrix $G$ and $g^{ij}$ are the coefficients of $G^{-1}$. Then, differentiating $GG^{-1}=\mathrm{Id}$, 
we find that, at $p$, $\nabla_Z g^{ij}=-\nabla_Z g_{ij}$. This shows that
\begin{equation*}
-\nabla_Z(H\cdot Z)=(\nabla_Z g_{ij})\nabla_{v_i}v_j\cdot Z-g^{ij}\nabla_Z(\nabla_{v_i}v_j\cdot Z).
\end{equation*}
Now notice
\begin{align*}
\nabla_Zg_{ij}&=\nabla_Z(v_i\cdot v_j)=(\nabla_Zv_i)\cdot v_j+v_i\cdot(\nabla_Z v_j)\\
&=(\nabla_{v_i}Z)\cdot v_j+v_i\cdot(\nabla_{v_j}Z)\\
&=-Z\cdot\nabla_{v_i}v_j-Z\cdot\nabla_{v_j}v_i=-2Z\cdot\nabla_{v_i}v_j,
\end{align*}
using the symmetry of the second fundamental form. This leads to the following conclusion: at $p$,
$$-\nabla_Z(H\cdot Z)=-2(\nabla_{v_i}v_j\cdot Z)^2-\nabla_Z(\nabla_{v_i}v_i\cdot Z).$$
\item  At $p$ and using $[Z,v_i]=0$,
\begin{align*}
R(v_i,Z)Z\cdot v_i&=-R(v_i,Z)v_i\cdot Z\\
&=-\nabla_{v_i}\nabla_Z v_i\cdot Z+\nabla_Z\nabla_{v_i}v_i\cdot Z\\
&=-\nabla_{v_i}\nabla_{v_i}Z\cdot Z+\nabla_Z(\nabla_{v_i}v_i\cdot Z)-\nabla_{v_i}v_i\cdot\nabla_ZZ\\
&=-\nabla_{v_i}(\nabla_{v_i}Z\cdot Z)+\nabla_{v_i}Z\cdot\nabla_{v_i}Z+\nabla_Z(\nabla_{v_i}v_i\cdot Z)\\
&\ \ -(\nabla_{v_i}v_i)^\top\cdot(\nabla_ZZ)^\top-(\nabla_{v_i}v_i)^\perp\cdot(\nabla_ZZ)^\perp.
\end{align*}
The latter two terms vanish because $H=0$ and the coordinates are normal. We conclude that
\begin{align*}
R(v_i,Z)Z\cdot v_i&=-\nabla_{v_i}(\nabla_Z v_i\cdot Z)+(\nabla_{v_i}Z)^\top\cdot(\nabla_{v_i}Z)^\top+(\nabla_{v_i}Z)^\perp\cdot(\nabla_{v_i}Z)^\perp\\
&\ \ +\nabla_Z(\nabla_{v_i}v_i\cdot Z)\\
&=\div_\Sigma((\nabla_ZZ)^\top)+(\nabla_{v_i}Z\cdot v_j)^2+(\nabla_{v_i}Z\cdot f_j)^2+\nabla_Z(\nabla_{v_i}v_i\cdot Z).
\end{align*}
\end{enumerate}
Substituting the second formula into the first proves the proposition.
\end{proof}

\begin{remark*}
One can also study the variation of the volume forms before integrating them. If we collect all terms appearing in the above calculations and, as above, we write $\iota_t^*\vol_t=f(t)\iota_0^*\vol_0$, we find
\begin{align*}
f''(0)=&-g(\nabla_{e_i}e_j,Z)^2-g(R(e_i,Z)Z,e_i)+\mbox{div}_\Sigma((\nabla_Z Z)^\top)\\
&-g(H,(\nabla_ZZ)^\perp)+g((\nabla_{e_i}Z), f_j)^2+g(H,Z)^2,
\end{align*}
where we assume $Z$ normal (but $\Sigma_0$ not necessarily minimal).
\end{remark*}

\paragraph{The second variation formula for coassociatives.} 

Now assume $M$ is endowed with a closed $G_2$ structure. Assume the initial submanifold $\Sigma=\Sigma_0$ is coassociative; for our first calculations there is no need to assume that it is also minimal.
Our goal is to work out a new expression for the second variation formula, adapted to this set-up. 

Restricting our attention to normal variations $Z$, we find

\begin{align*}
f''(0)\,\iota_0^*\vol_0&=\frac{d^2}{dt^2}\sqrt{\psi^2+|\mathcal{C}|^2}_{|t=0}\,\iota_0^*\vol_0=(\ddot\psi+\mathcal{C}_Z\cdot\mathcal{C}_Z)\,\iota_0^*\vol_0\\
&=\mathcal{L}_Z\mathcal{L}_Z\psi_{|\Sigma}+|\mathcal{C}_Z|^2\iota_0^*\vol_0\\
&=(d\iota_Z+\iota_Zd)(d\iota_Z+\iota_Zd)\psi_{|\Sigma}+|\mathcal{C}_Z|^2\iota_0^*\vol_0\\
&=d\iota_Zd\iota_Z\psi_{|\Sigma}+\iota_Zd\iota_Zd\psi_{|\Sigma}+|\mathcal{C}_Z|^2\iota_0^*\vol_0.
\end{align*}
The term $d\iota_Zd\iota_Z\psi$, restricted to $\Sigma$, vanishes under integration. 
Alternatively, it disappears by the calculation $\iota_Zd\iota_Z\psi_{|\Sigma}=0$.  Indeed, let $X_1,X_2,X_3\in T\Sigma$. 
We may assume $[Z,X_i]=0$, for $i=1,2,3$. Then, using the invariant formula for the exterior derivative of $\iota_Z\psi$, we find
\[
\iota_Zd\iota_Z\psi(X_1,X_2,X_3)=d\iota_Z\psi(Z,X_1,X_2,X_3) =\dots= Z(\iota_Z\psi(X_1,X_2,X_3)) = 0,
\]
since $\psi$ vanishes under contraction with one normal (associative) and three tangent (coassociative) vectors. 
We shall perform a similar calculation, in more detail, for $\iota_Zd\iota_Z\phi$ in the proof of Lemma \ref{lem:tau2}.

Furthermore,
\begin{align*}
\iota_Zd\iota_Zd\psi&=\iota_Zd\iota_Z(\tau_2\wedge\phi)\\
&=\iota_Zd(\iota_Z\tau_2\wedge\phi+\tau_2\wedge\iota_Z\phi)\\
&=\iota_Z((d\iota_Z\tau_2)\wedge\phi-\iota_Z\tau_2\wedge d\phi+d\tau_2\wedge\iota_Z\phi+\tau_2\wedge d\iota_Z\phi)\\
&=\iota_Zd\iota_Z\tau_2\wedge\phi+d\iota_Z\tau_2\wedge\iota_Z\phi+\iota_Zd\tau_2\wedge\iota_Z\phi+\iota_Z\tau_2\wedge d\iota_Z\phi\\
&\ \ \ +\tau_2\wedge\iota_Zd\iota_Z\phi,
\end{align*}
where we use the facts $d\phi=0$, $\iota_Z\iota_Z\phi=0$.

Substituting this into the original calculation and using $\phi_{|\Sigma}=0$, we find
\begin{align*}
f''(0)\,\iota_0^*\vol_0&=d\iota_Z\tau_2\wedge\iota_Z\phi_{|\Sigma}+\iota_Zd\tau_2\wedge\iota_Z\phi_{|\Sigma}
+\iota_Z\tau_2\wedge d\iota_Z\phi_{|\Sigma}\\
&\ \ \ +\tau_2\wedge\iota_Zd\iota_Z\phi_{|\Sigma}+|\mathcal{C}_Z|^2\iota_0^*\vol_0\\
&=d(\iota_Z\tau_2\wedge\iota_Z\phi)_{|\Sigma}+2\iota_Z\tau_2\wedge d\iota_Z\phi_{|\Sigma}+\iota_Zd\tau_2\wedge \iota_Z\phi_{|\Sigma}\\
&\ \ \ +\tau_2\wedge\iota_Zd\iota_Z\phi_{|\Sigma}+|\mathcal{C}_Z|^2\iota_0^*\vol_0.
\end{align*}  

\begin{lem}\label{lem:tau2}
Assume $\Sigma$ is coassociative and minimal. Then  
$$\tau_2\wedge\iota_Zd\iota_Z\phi_{|\Sigma}=\tau_2\wedge\gamma_{Z|\Sigma},$$
where $\gamma_Z\in\Lambda^2_-(\Sigma)$ is defined in terms of the second fundamental form: 
$$\gamma_Z(X_1,X_2)\coloneqq\iota_Z\phi((\nabla_{X_1} Z)^\top,X_2)+\iota_Z\phi(X_1,(\nabla_{X_2} Z)^\top).$$
\end{lem}
\begin{proof}
Given $\alpha\in\Lambda^2(M)$, recall the formula
\begin{align*}
d\alpha(X_0,X_1,X_2)&=X_0\alpha(X_1,X_2)-X_1\alpha(X_0,X_2)+X_2\alpha(X_0,X_1)\\
&\quad -\alpha([X_0,X_1],X_2)+\alpha([X_0,X_2],X_1)-\alpha([X_1,X_2],X_0).
\end{align*}
Let us apply it to $\alpha\coloneqq\iota_Z\phi$, so as to calculate $d\iota_Z\phi$. Calculating $\iota_Zd\iota_Z\phi$ corresponds to choosing 
$X_0\coloneqq Z$. 
Restricting to $\Sigma$ means that we must choose $X_1,X_2\in T\Sigma$. In our situation we may assume $[Z,X_1]=[Z,X_2]=0$. 
Only one of the above terms is then non-vanishing, so
\begin{align*}
\iota_Zd\iota_Z\phi(X_1,X_2)&=d\iota_Z\phi(Z,X_1,X_2)\\
&=Z\,\iota_Z\phi(X_1,X_2)\\
&=Z\,\phi(Z,X_1,X_2)\\
&=(\nabla_Z\phi)(Z,X_1,X_2)+\phi(\nabla_ZZ,X_1,X_2)+\phi(Z,\nabla_ZX_1,X_2)\\
&\ \ +\phi(Z,X_1,\nabla_ZX_2)\\
&=(\nabla_Z\phi)(Z,X_1,X_2)+\phi((\nabla_ZZ)^\perp,X_1,X_2)\\
&\ \ +\phi(Z,(\nabla_{X_1}Z)^\top,X_2)+\phi(Z,X_1,(\nabla_{X_2}Z)^\top)\\
&=\iota_Z(\nabla_Z\phi)(X_1,X_2)+\phi((\nabla_ZZ)^\perp,X_1,X_2)+\gamma_Z(X_1,X_2),
\end{align*}
where we use the fact that $\phi$ vanishes under contraction with one tangent (coassociative) and two normal (associative) vectors 
or with three tangent (coassociative) vectors. Notice the appearance of the second fundamental form of $\Sigma$, making $\gamma_Z$ tensorial in $Z$. 

The calculations above have generated three terms. We now need to wedge them with $\tau_2$. 
Recall that minimality implies $\tau_{2|\Sigma}\in\Lambda^2_-(\Sigma)$. Let us consider the three terms in turn.

\begin{enumerate}[1.]
\item Consider the term $\iota_Z\nabla_Z\phi=Z\lrcorner(\tau_2(Z)^\sharp\lrcorner\psi)$. 
The first contraction is with respect to a normal (associative) vector. 
Restricting to $\Sigma$ means that we will be further contracting $\psi$ with two tangent (coassociative) vectors. 
The remaining contraction must thus again be with a normal (associative) vector, i.e.
$$\iota_Z\nabla_Z\phi_{|\Sigma}=Z\lrcorner((\tau_2(Z)^\sharp)^\perp\lrcorner\psi)_{|\Sigma}.$$
Recall from the model expression of $\psi$ that the result lies in $\Lambda^2_+(\Sigma)$, so this term vanishes after wedging.

\item The second term is of the form $\Lambda^2_+(\Sigma)$, so it vanishes after wedging.

\item To conclude, we need to show that $\gamma_Z\in\Lambda^2_-(\Sigma)$. 
This is a pointwise computation which we can perform on $\R^4$ with the standard structures. It is a special case of the following general statement. 

Let $\alpha\in\Lambda^2_+(\R^4)$ and $f\in\End(\R^4)$ be symmetric and trace-free. Then
$$\alpha_f(X_1,X_2)\coloneqq\alpha(f(X_1),X_2)+\alpha(X_1,f(X_2))$$
is anti-selfdual.

To prove this let $E_{12+34}$, $E_{13-24}$, $E_{14+23}$ be the matrices corresponding to the standard basis of $\Lambda^2_+(\R^4)$. 
The matrix $A$ representing $\alpha$ is a linear combination of these. Let $M$ be the matrix representing $f$. 
It then suffices to check that $MA+AM$ is a linear combination of the matrices corresponding to the standard basis of $\Lambda^2_-(\R^4)$.
\end{enumerate}
\end{proof}

Integration leads to the following conclusion.

\begin{thm}\label{thm:secvar}
Let $M$ be endowed with a closed $G_2$ structure. Assume $\Sigma_0$ is coassociative and minimal. Then, for any normal variation $Z$,
\begin{equation*}\label{eq:secvar}
\frac{d^2}{dt^2}\Vol(\Sigma_t)_{|t=0}=\int_\Sigma 2\iota_Z\tau_2\wedge d\iota_Z\phi+\iota_Zd\tau_2\wedge \iota_Z\phi+\tau_2\wedge\gamma_Z
+|\mathcal{C}_Z|^2\iota_0^*\vol_0.
\end{equation*}
In particular, assume the variation takes place within the coassociative moduli space $\mathcal{M}$, i.e., $Z\in T_\Sigma\mathcal{M}$. Then 
\begin{equation*}\label{eq:moduli_secvar}
\frac{d^2}{dt^2}\Vol(\Sigma_t)_{|t=0}=\int_\Sigma \iota_Zd\tau_2\wedge \iota_Z\phi+\tau_2\wedge\gamma_Z.
\end{equation*}
\end{thm}
 
 \medskip
 
The two terms $d\tau_2$, $\mathcal{C}_Z$ admit alternative formulations, as follows.

\begin{enumerate}[1.]
\item Recall from \cite{Bryant} the existence of a pointwise $G_2$-equivariant isomorphism
\begin{equation}\label{eq:i_map}
\i: S^2\rightarrow \Lambda^3_1\oplus\Lambda^3_{27},  
\end{equation}
defined on the space $S^2$ of symmetric 2-tensors. 
Viewed on $\R^7$, it admits a compact expression in terms of Bryant's notation, \cite[eq.~(2.17)]{Bryant}: given $h=\sum_{i,j}h_{ij}\,e^i\otimes e^j\in S^2$,  
\[ 
\begin{split}
\i(h)	&=	\sum_{r,j,k,l}\epsilon_{rkl}h_{rj}\,e^j\wedge e^k\wedge e^l\\ 
	&= \sum_{r,j,k,l} \frac13 \left(\epsilon_{rkl}h_{rj} + \epsilon_{rjk}h_{rl}+\epsilon_{rlj}h_{rk}\right) e^j\wedge e^k\wedge e^l \\
	& = 2\sum_{j<k<l}\sum_{r}\left(\epsilon_{rkl}h_{rj} + \epsilon_{rjk}h_{rl}+\epsilon_{rlj}h_{rk}\right) e^j\wedge e^k\wedge e^l.
\end{split} 
\]
This formula shows that $\i(g_\phi)=6\phi$. 
Moreover, the image of the subspace $S^2_0$ of trace-free symmetric 2-tensors coincides with $\Lambda^3_{27}$.

Recall also the following two formulae, respectively \cite{Bryant} eq.~(4.36) and (4.39), valid when $d\phi=0$. 
The first describes the scalar curvature of $M$, showing it is non-positive:
\begin{align*}
\tr_g(\Ric)&=-\frac12|\tau_2|^2,\\
d\tau_2&=\frac{3}{14}|\tau_2|^2\phi+\frac12\star(\tau_2\wedge\tau_2)-\frac12\i(\Ric_0),
\end{align*}
where $\Ric_0$ is the trace-free Ricci tensor of $M,$ i.e., $\Ric=\Ric_0+\tfrac17\tr(\Ric)g$. It follows that $\i(\Ric)=\i(\Ric_0)+\tfrac67\tr(\Ric)\phi$. 
Combining these formulae we find
\begin{equation}\label{eq:dtau2}
d\tau_2=\frac12\star(\tau_2\wedge\tau_2)-\frac12\i(\Ric).
\end{equation}
If we plug equation \eqref{eq:dtau2} into our second variation formula, 
we obtain an alternative expression which emphasizes the role of the ambient Ricci curvature. 

\item Recall that  $\mathcal{C}_Z$ was defined in the proof of Lemma \ref{l:firstvar_G2}. 
This term already appears in \cite{LeVanzura}. In their Lemma 3.13 (which holds also in the case $d\phi=0$, without assuming $d\psi=0$) it is shown that 
\begin{equation}\label{eq:CZ}
\mathcal{C}_Z\coloneqq \left|\nabla_Z(\mathcal{C}(\iota_{t^*}v_1,\dots,\iota_{t^*}v_4))_{|t=0}\right|^2=|d\iota_Z\phi_{|\Sigma}|^2.
\end{equation}
If one also assumes $d\psi=0$, then this is the only term appearing in the second variation formula so our formula coincides with that of \cite{McLean} 
and confirms that coassociative submanifolds are stable minima, as expected.
\end{enumerate}

\begin{remark*}
Notice that the second variation formula becomes tensorial with respect to $Z$ when restricted to coassociative variations. 
This is related to the fact that the moduli space is finite-dimensional, so these variations cannot be perturbed via arbitrary functions. 
\end{remark*}

\paragraph{The second variation formula for Lagrangians.}The appearance of the ambient Ricci tensor is an interesting fact.
 It replaces the curvature term which appears, without taking any compact geometric form, in the standard second variation formula. 
 An analogous situation occurs also in \cite{Oh} in a way that is perhaps more transparent than above.

Assume $(M^{2n},J,\omega)$ is K\"ahler and $\Sigma^n\hookrightarrow M$ is Lagrangian, i.e., $\omega_{|\Sigma}\equiv 0$. 
The K\"ahler form then provides an isomorphism
\begin{equation*}
T\Sigma^\perp\simeq\Lambda^1(\Sigma), \ \ Z\mapsto\zeta\coloneqq\iota_Z\omega.
\end{equation*}
Lagrangian submanifolds are not calibrated, so Oh's context is closer in spirit to the general Riemannian situation than to ours. 
On the other hand, notice the analogies with $\phi$, with the coassociative condition $\phi_{|\Sigma}\equiv 0$ 
and with the isomorphism $T\Sigma^\perp\simeq\Lambda^2_+(\Sigma)$. 

Oh's second variation formula is as follows. Assume $\Sigma_0$ is minimal Lagrangian. Consider any normal variation $Z$. Then
$$\frac{d^2}{dt^2}(\Vol(\Sigma_t))_{|t=0}=\int_\Sigma\left((\Delta \zeta,\zeta)-\Ric(Z,Z)\right)\iota_0^*\vol_0,$$
where $\Delta$ denotes the Hodge Laplacian on $\Sigma_0$ and $\Ric$ is the ambient Ricci curvature. 
In particular, it shows that $\Sigma_0$ is stable if $\Ric\leq 0$.

The proof consists in rearranging the terms in the standard second variation formula. Using normal coordinates on $\Sigma_0$,
\begin{align*}
(\nabla_{e_i}Z\cdot f_j)^2&=(\nabla_{e_i}^\perp Z,\nabla_{e_i}^\perp Z)\\
&=\nabla_{e_i}(\nabla_{e_i}^\perp Z,Z)-(\nabla_{e_i}^\perp\nabla_{e_i}^\perp Z,Z)\\
&=(1/2)\Delta(|Z|^2)-(\Delta^\perp Z,Z).
\end{align*}
The first term on the RHS vanishes by integration by parts. The Weitzenb\"ock identity on $\Sigma$ shows that 
$\Delta\zeta=-\omega(\Delta^\perp Z,\cdot)+\Ric_\Sigma(JZ,\cdot)$. Evaluating this on $JZ$ we find
\begin{align*}
-(\Delta^\perp Z,Z)&=\Delta\zeta\cdot\zeta-\Ric_\Sigma(JZ,JZ)\\
&=\Delta\zeta\cdot\zeta-R_\Sigma(e_i,JZ)JZ\cdot e_i.
\end{align*}
The Gauss equation for curvature, together with $\nabla J=0$ and $H=0$, yields
\begin{align*}
R(e_i,JZ)JZ\cdot e_i&=R_\Sigma(e_i,JZ)JZ\cdot e_i+|(\nabla_{e_i}JZ)^\perp|^2+(\nabla_{e_i}{e_i})^\perp\cdot(\nabla_{JZ}JZ)^\perp\\
&=R_\Sigma(e_i,JZ)JZ\cdot e_i+|(\nabla_{e_i}Z)^\top|^2.
\end{align*}
Also recall that, since $M$ is K\"ahler, $R(e_i,JZ)JZ\cdot e_i=R(Je_i,Z)Z\cdot Je_i$. 
Comparing this with the standard second variation formula allows us to complete the curvature term so as to obtain $-\Ric(Z,Z)$ 
and to cancel the term depending on the second fundamental form.

We may apply these same calculations to the formula found above for $f''(0)$, corresponding to the second derivative of the volume form. 
We then find that, for $\Sigma_0$ minimal Lagrangian and using normal variations,
$$f''(0)=(1/2)\Delta(|Z|^2)+(\Delta \zeta,\zeta)-\Ric(Z,Z)+\div_{\Sigma}((\nabla_ZZ)^\top).$$

\paragraph{Conclusions.}
Let us take a moment to summarize analogies and differences between the three second variation formulae discussed in this section: 
\begin{enumerate}[1.]
\item The ambient manifolds for the Riemannian and the Lagrangian formulae are automatically torsion-free, so the torsion tensor plays no role. 
Our formula makes precise the role played by torsion in the $G_2$ context.
\item The coassociative formula, via equations (\ref{eq:dtau2}), (\ref{eq:CZ}), can be summarized as 
$$\frac{d^2}{dt^2}(\Vol(\Sigma_t))_{|t=0}=\int_{\Sigma}(\mbox{Laplacian-Ricci+torsion}).$$
Up to torsion, this is remarkably similar to the Lagrangian formula. 
\item Both the Riemannian and the coassociative formulae show the special role played by totally geodesic submanifolds.
\end{enumerate}
Furthermore, the coassociative formula is particularly well-adapted to the special geometric features of $G_2$ geometry; 
in particular, to the existence of moduli spaces.

\section{Controlling the torsion}\label{secTorsion}

We wish to find geometric situations in which our second variation formula, Theorem \ref{thm:secvar}, 
produces a non-negative (or positive) result even in the presence of torsion: this will mean that the submanifold $\Sigma_0$ is a (strict) local minimum point for the volume functional. 

As already seen, we can cancel the term $d\iota_Z\phi$ by moving within the coassociative moduli space. 
We can cancel the term $\gamma_Z$ by imposing that the submanifold be totally geodesic. 
The main question is thus how to control the term $\iota_Z d\tau_2\wedge\iota_Z\phi$. In essence, this entails controlling $d\tau_2$.

\smallskip

In \cite[eq.~(4.65)]{Bryant}, Bryant introduces the class of $G_2$ structures satisfying a certain ``quadratic condition". 
This is a strong constraint, but we shall show below that it provides a good framework within which to enforce our desired positivity. 
It relies on the following linear-algebraic construction.

Given any $\beta\in\Lambda^2_{14}$, \cite[Section 2.7.4]{Bryant} shows that:
\begin{enumerate}[1.]
\item $\star(\beta\wedge\beta)\in\Lambda^3_1\oplus\Lambda^3_{27}$. Its irreducible decomposition is
\[
\star(\beta\wedge\beta)=-\frac17|\beta|^2\phi+\left(\frac17|\beta|^2\phi+\star(\beta\wedge\beta)\right).
\]
\item Set $\gamma_\beta\coloneqq \tfrac17|\beta|^2\phi+\star(\beta\wedge\beta)\in\Lambda^3_{27}$, depending quadratically on $\beta$. Then 
$$|\gamma_\beta|=\sqrt{\tfrac67}\,|\beta|^2.$$
\end{enumerate}
Furthermore, \cite[eq.~(4.33)]{Bryant} shows that $d\phi=0$ implies that $d\tau_2\in\Lambda^3_1\oplus\Lambda^3_{27}$. 
It decomposes as $d\tau_2=\tfrac17|\tau_2|^2\phi+\gamma$, 
for some $\gamma\in\Lambda^3_{27}$. Bryant's condition concerns the situation where $\gamma$ arises from the above construction, 
applied to the case $\beta=\tau_2$.

\begin{definition}
Let $M$ be endowed with a closed $G_2$ structure $\phi$ with non-vanishing torsion. 
We say that the $G_2$ structure satisfies the {\em quadratic condition} if there exists $\lambda\in\R$ such that the irreducible decomposition of $d\tau_2$ 
is of the form
\begin{equation*}
d\tau_2=\frac17|\tau_2|^2\phi+\lambda\,\gamma_\tau,
\end{equation*}
where, as above, $\gamma_\tau\coloneqq \frac17|\tau_2|^2\phi+\star(\tau_2\wedge\tau_2)$.
\end{definition}

Given $\beta\in\Lambda^2_{14}$ and $\lambda\in\R$, we are thus interested in the algebraic properties of the 3-form on $\R^7$
\[
\gamma_{\lambda,\beta}\coloneqq \frac17|\beta|^2\phi+\lambda\,\gamma_\beta\in\Lambda^3_1\oplus\Lambda^3_{27}.
\]
 
\begin{prop}\label{prop:estimate}
Assume $|\lambda|\leq\tfrac{1}{\sqrt{21}}$. Choose any $\beta\in\Lambda^2_{14}$ and any coassociative $4$-plane $\pi$ in $\R^7$. Then the bilinear form 
$$B_{\lambda,\beta}^\pi:\pi^\perp\times\pi^\perp\rightarrow\Lambda^4(\pi), \ \ 
B_{\lambda,\beta}^\pi(Z_1,Z_2)\coloneqq\iota_{Z_1}\gamma_{\lambda,\beta}\wedge\iota_{Z_2}\phi_{|\pi},
$$
is non-negative: $B_{\lambda,\beta}^\pi(Z,Z)\geq 0$, for all $Z\in\pi^\perp$.
\end{prop}

\begin{proof}
Write
\[
B_{\lambda,\beta}^\pi(Z,Z) = \frac17|\beta|^2\iota_Z\phi\wedge\iota_Z\phi_{|\pi}+\lambda(\iota_Z\gamma_\beta\wedge\iota_Z\phi_{|\pi}).
\]
Using the $G_2$ action (which will send $\beta$ to some other $\beta'$, with $|\beta|=|\beta'|$), we may assume $Z=|Z|e_1$, $\pi=4567$. Both terms are then multiples of $\vol_{4567}$. In particular, 
\begin{align*}
\iota_Z\phi_{|\pi}\wedge\iota_Z\phi_{|\pi}&=|Z|^2(45+67)\wedge(45+67)=2|Z|^2\vol_{4567},\\ 
|\iota_Z\gamma_{\beta'|\pi}\wedge\iota_Z\phi_{|\pi}|&\leq |Z||\gamma_{\beta'}||\iota_Z\phi_{|\pi}|=|Z|^2\sqrt{\frac67}|\beta|^2\sqrt{2}.
\end{align*}
Non-negativity is thus ensured by the condition
\[
\frac27 |Z|^2|\beta|^2\geq \sqrt{\frac{12}{7}}|\lambda||Z|^2|\beta|^2.
\]
\end{proof}

Applying this to $\beta\coloneqq\tau_2$, thus $\gamma_{\lambda,\beta}=d\tau_2$, we obtain the following result. 

\begin{cor}\label{cor:stability}
Let $M$ be endowed with a closed $G_2$ structure satisfying the quadratic condition with $|\lambda|<1/\sqrt{21}$. Assume $\Sigma_0$ is coassociative and totally geodesic. Then $\Sigma_0$ is strictly stable within the coassociative moduli space.
\end{cor}

\begin{remark*}
Of course, the moduli space might consist of a single point: as mentioned, its dimension is determined by the topology of $\Sigma$. 
In this case the conclusion is trivial.
\end{remark*}

The bilinear forms seen in Proposition \ref{prop:estimate} can be generalized to any element in $\Lambda^3_1\oplus\Lambda^3_{27}$. The isomorphism with $S^2$ defined by equation (\ref{eq:i_map}) then yields the following, equivalent, expression.

\begin{prop}\label{prop:Bh}
Let $h$ be a symmetric $2$-tensor on $\R^7$. Let $\pi$ be a coassociative $4$-plane. Consider the bilinear form $B_h^\pi:\pi^\perp\times\pi^\perp\rightarrow\Lambda^4(\pi)$,
\[
B_h^\pi(Z_1,Z_2)\coloneqq\iota_{Z_1}\i(h)\wedge\iota_{Z_2}\phi_{|\pi}.
\]
Then, for any $Z_1,Z_2\in\pi^\perp$,
\begin{equation*}
B_h^\pi(Z_1,Z_2)=(4h(Z_1,Z_2)+2\tr(h_{|\pi})g(Z_1,Z_2))\vol_{\pi}.
\end{equation*}
\end{prop}
\begin{proof}
As a first step, let us assume $\pi=4567$. We shall write $h=\sum_{i,j}h_{ij}\,e^i\otimes e^j$ and $Z=z_1e_1+z_2e_2+z_3e_3$.

Choose, for example, $Z_1= e_1$, $Z_2 = e_i$. Using the symmetry of $h$ one finds:
\begin{align*}
\iota_{e_1}\i(h)=&+(h_{11}+2h_{44})e^4\wedge e^5+(-h_{11}-2h_{55})e^5\wedge e^4\\
&+(h_{11}+2h_{66})e^6\wedge e^7+(-h_{11}-2h_{77})e^7\wedge e^6\\
&+(h_{12}-2h_{47})e^4\wedge e^6+(-h_{12}-2h_{56})e^6\wedge e^4\\
&+(-h_{12}+2h_{56})e^5\wedge e^7+(h_{12}+2h_{47})e^7\wedge e^5\\
&+(-h_{13}+2h_{46})e^4\wedge e^7+(h_{13}-2h_{57})e^7\wedge e^4\\
&+(-h_{13}-2h_{57})e^5\wedge e^6+(h_{13}+2h_{46})e^6\wedge e^5.
\end{align*}
Using our previous short-hand notation, this implies
\begin{align*}
\iota_{e_1}\i(h)\wedge \iota_{e_1}\phi_{|4567}&=(4h_{11}+2(h_{44}+h_{55}+h_{66}+h_{77}))\,4567,\\
\iota_{e_1}\i(h)\wedge \iota_{e_2}\phi_{|4567}&=(4h_{12})\,4567,\\
\iota_{e_1}\i(h)\wedge \iota_{e_3}\phi_{|4567}&=(4h_{13})\,4567.
\end{align*}
Similar calculations hold for any $Z_1,Z_2\in\{e_1,e_2,e_3\}$. 

We can reduce the general case to the case above via the $G_2$ action (which will send $h$ to some other $h'$).
\end{proof}

\begin{remark}
The previous result shows, in particular, that the bilinear form $B_h^\pi$ is symmetric. This can alternatively be proved directly from the definition, as follows. 

Let us write $\i(h) = a\,\phi + \gamma\in \Lambda^3_1\oplus\Lambda^3_{27}$, for some $a\in\R$ and $\gamma\in\Lambda^3_{27}$. Then
\[
B_h^\pi(Z_1,Z_2) 	= a\iota_{Z_1}\phi\wedge\iota_{Z_2}\phi_{|\pi} +  \iota_{Z_1}\gamma\wedge\iota_{Z_2}\phi_{|\pi}
				= a\iota_{Z_2}\phi\wedge\iota_{Z_1}\phi_{|\pi} +  \iota_{Z_1}\gamma\wedge\iota_{Z_2}\phi_{|\pi}.
\]
It thus suffices to show that
\[
i_{Z_1}\gamma \wedge i_{Z_2}\phi_{|\pi} = i_{Z_2}\gamma \wedge i_{Z_1}\phi_{|\pi}. 
\]
From $\gamma\wedge\phi=0$, we obtain
\[
\begin{split}
0 	&= i_{Z_2}(i_{Z_1}(\gamma\wedge\phi)) \\
	&= i_{Z_2}(i_{Z_1}\gamma) \wedge\phi + i_{Z_1} \gamma \wedge i_{Z_2}\phi 
		- i_{Z_2} \gamma \wedge i_{Z_1}\phi + \gamma \wedge i_{Z_2} (i_{Z_1}\phi).
\end{split}
\]
Since both the first and the last summand vanish when restricted to the coassociative plane $\pi$, the required identity follows. 
\end{remark}

\section{Example: ERP structures}\label{s:erp}

The following special case of the quadratic condition provides a good testing ground for the above results.

\begin{definition}\label{def:erp}
A closed $G_2$ structure is {\em Extremally Ricci-Pinched} (ERP) if it satisfies the quadratic condition with $\lambda=\tfrac16$; equivalently,
\[
d\tau_2=\frac16\left(|\tau_2|^2\phi+\star(\tau_2\wedge\tau_2)\right).
\]
\end{definition}

In the compact case, an alternative way of introducing these structures stems from the fact that, for any closed $G_2$ structure,
\[
\int_M|\Ric_0|^2\vol\geq \frac{4}{21}\int_M|s|^2\vol,
\]
where $s$ denotes the scalar curvature.
One can show that the equality is equivalent to the quadratic condition with $\lambda=\tfrac16$, see \cite[Remark 13]{Bryant}. 
This explains the above terminology.

\begin{remark*} 
\cite[Remark 14]{Bryant}  shows that $\lambda=\tfrac16$ is the only possible value for a $G_2$ structure satisfying the quadratic condition 
on a compact manifold. 
\end{remark*}

\smallskip

Notice that $1/6<1/\sqrt{21}$: ERP manifolds thus fall within the range of Corollary \ref{cor:stability}.

Let us restrict our attention to ERP structures on compact manifolds; more generally, to ERP structures admitting a compact quotient. By \cite[Theorem 3.7]{Bryant}, such structures have {\em special torsion of positive type} in the sense of \cite{Ball}, i.e., $\tau_2^3=0$. 
The torsion is then modeled on one of the exceptional orbits of the $G_2$ action on $\Lambda^2_{14}$: specifically, up to the $G_2$ action at each point, we may assume that $\tau_2$ is of the form $c\beta_+$, where $\beta_+=e^{45}-e^{67}$ and 
$c\in\R\smallsetminus\{0\}$. 
Notice that 
\[
|\tau_2|^2=|c\beta_+|^2=2c^2, \qquad \star(\tau_2\wedge\tau_2)=\star(c\beta_+\wedge c\beta_+)=-2c^2 e^{123}.
\]
The ERP condition then implies that
\[
d\tau_2 = \frac{c^2}{3}\left(e^{145}+e^{167}+e^{246}-e^{257}-e^{347}-e^{356}\right). 
\]
Choosing, for instance, the coassociative $4$-plane $\pi=4567$, we see that the bilinear form 
$B^\pi_{1/6,\tau_2}:\pi^\perp\times\pi^\perp\to \Lambda^4\pi\cong\R$ 
introduced in Proposition \ref{prop:estimate} is given by 
\[
B^\pi_{1/6,\tau_2} = \frac23\,c^2\left(e^1\otimes e^1 + e^2\otimes e^2 + e^3\otimes e^3 \right). 
\]
This form is positive definite, as expected from Proposition \ref{prop:estimate}.

\begin{example}\label{BryEx} 
The first example of ERP manifold of this type is due to Bryant \cite{Bryant}, who described it in terms of the homogeneous space 
$M=(\mathrm{SL}(2,\C)\ltimes \C^2)/\mathrm{SU}(2)$. We shall adopt the equivalent description \cite{ClIv}, as a Lie group. 

Consider the seven-dimensional, simply connected, solvable Lie group
\[
\mathrm{G}\coloneqq \left\{
\begin{pmatrix}
\exp(t)	&	z		&	x\\
0		&	\exp(-t)	&	y\\
0		&	0		&	1
\end{pmatrix}
~|~ t\in\R,~x,y,z\in\C
\right\}
\cong \mathrm{Sol}_3\ltimes \C^2. 
\]
Its Lie algebra is isomorphic to a semidirect product $\mathfrak{g} \cong \mathfrak{s}\ltimes_\mu\mathfrak{h}$, 
where $\mathfrak{s}\cong \mathrm{Lie}(\mathrm{Sol}_3)$, 
$\mathfrak{h}$ is a $4$-dimensional Abelian ideal, and $\mu:\mathfrak{s}\rightarrow\mathrm{Der}(\mathfrak{h})\cong\mathrm{End}(\R^4)$ 
is a Lie algebra homomorphism. 

More specifically, we may choose the following basis of $\mathfrak{s}$
\[
e_1 = 
\frac12
\begin{pmatrix}
1&0&0\\
0&-1&0\\
0&0&0
\end{pmatrix},
\quad 
e_2 = 
\frac12
\begin{pmatrix}
0&1&0\\
0&0&0\\
0&0&0
\end{pmatrix},
\quad
e_3 = 
\frac12
\begin{pmatrix}
0&i&0\\
0&0&0\\
0&0&0
\end{pmatrix},
\]
and the following basis of $\mathfrak{h}$
\[
e_4 = 
\begin{pmatrix}
0&0&0\\
0&0&1\\
0&0&0
\end{pmatrix},
~
e_5 = 
\begin{pmatrix}
0&0&0\\
0&0&i\\
0&0&0
\end{pmatrix},
~
e_6 = 
-\frac12
\begin{pmatrix}
0&0&i\\
0&0&0\\
0&0&0
\end{pmatrix},
~
e_7 = 
-\frac12
\begin{pmatrix}
0&0&1\\
0&0&0\\
0&0&0
\end{pmatrix}.
\]
The non-zero Lie brackets of $\mathfrak{s}$ are
\[
[e_1,e_2]_{\mathfrak{s}} = e_2,\quad [e_1,e_3]_{\mathfrak{s}}=e_3.
\]
The map $\mu:\mathfrak{s}\rightarrow\mathrm{Der}(\mathfrak{h})$ acts as follows:
\[
\begin{array}{llll}
\mu(e_1)(e_4) = -\frac12\,e_4,& \mu(e_1)(e_5) = -\frac12\,e_5,& \mu(e_1)(e_6) = \frac12\,e_6,& \mu(e_1)(e_7) = \frac12\,e_7, \\
\mu(e_2)(e_4) = -e_7,& \mu(e_2)(e_5) = -e_6,& \mu(e_2)(e_6) = 0,& \mu(e_2)(e_7) = 0,\\
\mu(e_3)(e_4) = -e_6,& \mu(e_3)(e_5) = e_7,& \mu(e_3)(e_6) = 0,& \mu(e_3)(e_7) = 0. 
\end{array}
\]
Notice that $\mathfrak{g}$ is a seven-dimensional non-unimodular completely solvable real Lie algebra. 
The Lie bracket $[\cdot,\cdot]$ on $\mathfrak{g}$ is defined as follows: for all $x,y\in\mathfrak{s}$ and $u,v\in\mathfrak{h}$,
\[
[x,y] = [x,y]_{\mathfrak{s}},\quad [x,u] = \mu(x)(u),\quad [u,v] =  [u,v]_{\mathfrak{h}}  = 0. 
\]

Let $\mathcal{B}^*=(e^1,\ldots,e^7)$ denote the dual basis of $\mathcal{B} = (e_1,\ldots,e_7)$.  
The structure equations $\left(de^i\right)_{i=1,\ldots,7}$ of $\mathfrak{g}$ are the following
\[
\left(0, -e^{12}, -e^{13}, \frac12\,e^{14}, \frac12\,e^{15}, -\frac12\,e^{16}+e^{25}+e^{34},  -\frac12\,e^{17}+e^{24}-e^{35}\right).
\]
It is now straightforward to check that the left-invariant 3-form 
\[
\phi = e^{123}+e^{1}\wedge(e^{45}+e^{67})+e^2\wedge(e^{46}-e^{57})-e^3\wedge(e^{47}+e^{56}), 
\]
defines an ERP structure on $\mathrm{G}$ with intrinsic torsion form 
\[
\tau_2 = 3\left(e^{45}-e^{67}\right).
\]
In particular, the basis $\mathcal{B}$ is $g_\phi$-orthonormal and $|\tau_2|^2=18$.  

Within this manifold we shall examine three different coassociative fibrations, exhibiting very different properties. 
The first is due to Bryant \cite{Bryant}. The other two are new.

\smallskip

\noindent1. Via the usual identification $\mathfrak{g}\cong T_{1_{\mathrm{G}}}\mathrm{G}$, 
the ideal $\mathfrak{h} = \langle e_4,e_5,e_6,e_7 \rangle \subset \mathfrak{g}$ corresponds to the tangent space at the identity $1_{\mathrm{G}}$ 
of an Abelian subgroup $\Sigma\cong \C^2$ of $\mathrm{G}$. 
Since $\phi_{|\Sigma}=0$, $\Sigma$ is a coassociative submanifold of $(\mathrm{G},\phi)$. 
Moreover, the restriction of the left-invariant metric $g_\phi$ to $\Sigma$ is flat. 

We can define coassociative deformations using left translations. 
We conclude that the map $\pi:\mathrm{G} \to \mathrm{G}/\mathrm{\C^2}$ is a coassociative fibration with flat fibres. 
It follows from \cite{Bryant} that $\Sigma$ can be compactified by means of a suitable lattice $\mathrm{L}\subset\Sigma$ 
that is preserved by $\mathrm{S}$. The quotient space then defines a flat coassociative $T^4$-fibration. 

In Section \ref{s:first} we discussed the fact that any fibration of this type admits the structure of a Riemannian submersion. 
Let us apply this construction to the example at hand.

In general, given any two Lie groups $\mathrm{S}$, $\mathrm{H}$ and an action of $\mathrm{S}$ on $\mathrm{H}$ by automorphisms, 
consider the Lie group $\mathrm{G}\coloneqq \mathrm{S}\ltimes \mathrm{H}$. 
As differentiable manifolds, we can identify $\pi:\mathrm{G}\rightarrow \mathrm{G}/\mathrm{H}$ 
with $\pi:\mathrm{S}\times \mathrm{H}\rightarrow \mathrm{S}$. 
Since $\mathrm{H}$ is normal in $\mathrm{G}$, the identification $\mathrm{G}/\mathrm{H}\simeq \mathrm{S}$ is also a group isomorphism. 
Given any $e\in T_{1_\mathrm{S}}\mathrm{S}$, i.e., $(e,0)\in T_{1_\mathrm{G}}\mathrm{G}$, let $e$ also denote the corresponding left-invariant 
vector field on $\mathrm{G}$. Let us write $e=d/dt(s_t)_{|t=0}$ for some curve $s_t$ in $\mathrm{S}$ such that $s_0=1_{\mathrm{S}}$,  
and let $g=(s,h)\in \mathrm{G}$. 
Then
\begin{align*}
e_{|(s,h)}	&=(L_{(s,h)})_*(e)=d/dt((s,h)\cdot(s_t,1_\mathrm{H}))_{|t=0}=d/dt(ss_t,h)_{|t=0}\\
		&=((L_s)_*(e),0).
\end{align*}
This shows that, under the above identifications, $e_{|g}\in T_g\mathrm{G}\simeq(e_{|s},0)\in T_s\mathrm{S}\oplus T_h\mathrm{H}$, 
so the vector field $e$ on $\mathrm{G}$ projects down to the left-invariant vector field $e$ on $\mathrm{S}$ 
defined by left-translation on $\mathrm{S}$.

In our case, we have endowed the total space of the fibration $M\coloneqq\mathrm{G}$ with the metric induced by $\phi$. 
The conclusion is that the construction of Section \ref{s:first} induces on the base $B\coloneqq \mathrm{G}/\C^2\simeq\mathrm{Sol}_3$ 
precisely the metric for which $e_1,e_2,e_3$ is an orthonormal basis.

\begin{remark*}
In general, the curvature properties of the metric $g$ on $\mathcal{M}$ are hard to establish. 
In this example, the algebraic structure on $\mathcal{M}\cong \mathrm{Sol}_3$ allows us to see that  $\mathrm{Ric}_g = -2\, g$.
\end{remark*}

\begin{remark*}
The above shows that the vector fields $e_1,e_2,e_3$ on $\mathrm{G}$ are the lifts, in the sense of Section \ref{s:first}, 
of the corresponding vector fields on $\mathrm{S}$. 
In the language of Section \ref{s:example} we may say that the lifted vector fields are ``horizontal" and that the horizontal distribution is integrable.

In this context there are two natural ways to use $e\in\mathfrak{s}$ to deform a coassociative fibre $\Sigma$: 
either using the flow of the corresponding vector field on $\mathrm{G}$, i.e., moving each $(1_S,h)\simeq p\in\Sigma$ horizontally, 
or via the group action: $(1_\mathrm{S},h)\mapsto\exp(te)\cdot (1_\mathrm{S},h)=R_h(\exp(te))$. 
These two deformations coincide, but only up to reparametrization of $\Sigma$: indeed, the infinitesimal deformation 
\[
Z_{|p}\coloneqq \left.\frac{d}{dt}\right|_{t=0}(R_h(\exp(te))= \left.\frac{d}{dt}\right|_{t=0}(\exp(te), \exp(te)\cdot h)
\]
has non-trivial vertical components, but $(Z_{|p})^\perp=e_{|p}$. In any case, they define the same objects in the moduli space $\mathcal{M}$. 
The same reasoning shows that $R_h$ does not act by isometries on $\mathrm{G}$ and that it does not preserve horizontal vectors. 
In particular, it does not preserve the horizontal leaves of the submersion.

In this sense, the Riemannian submersion structure is not adapted to the group structure.
\end{remark*}

We shall now compute the mean curvature vector field $H$ of the fibres. 
By left-invariance, it suffices to do this for the fibre $\Sigma$, at the point $1_\mathrm{G}$. The second fundamental form 
$\mathrm{II}:T_{1_\mathrm{G}}\Sigma\times T_{1_\mathrm{G}}\Sigma\to T_{1_\mathrm{G}}\Sigma^\perp$ can be computed 
via the Weingarten formula 
\[
 \mathrm{II}(e_i,e_j)\cdot \nu = -  \nabla_{e_i}\nu \cdot e_j, 
\]
where $\nu\in (T_{1_\mathrm{G}}\Sigma)^\perp \cong \mathfrak{s}$ and $\nabla$ denotes the Levi-Civita connection of $g_\phi$.  
Using the Koszul formula, we can rewrite the previous identity as follows 
\[
 \mathrm{II}(e_i,e_j) \cdot \nu 	= -  \nabla_{e_i}\nu \cdot e_j 
								= \frac12\left( \nu \cdot [e_i,e_j] +  e_j \cdot [\nu,e_i] +  e_i \cdot [\nu,e_j]  \right).
\]
Since $\mathfrak{h}$ is abelian, the first summand on the right hand side is zero. Furthermore, for $i=j$ the formula simplifies to 
\begin{equation}\label{IIexB}
\mathrm{II}(e_i,e_i) \cdot \nu  =  e_i \cdot [\nu,e_i]. 
\end{equation}
Notice that this implies 
\[
H \cdot \nu	 	= \sum_{i=4}^7   e_i \cdot [\nu,e_i]  
				= \sum_{i=4}^7  \mathrm{ad}_\nu e_i \cdot e_i = \mathrm{tr}\left({\mathrm{ad}_{\nu}}_{|\mathfrak{h}}\right) 
				= \mathrm{tr}(\mu(\nu)).    
\]
The definition of $\mu$ then implies that $H=0$, as expected by Section \ref{s:first} and the fact that 
$\tau_{2|\Sigma}= 3\left(e^{45}-e^{67}\right) \in\Lambda^2_-(\Sigma)$. 

On the other hand, $\Sigma$ is not totally geodesic: the second fundamental form is identically zero if and only if $\mu(\nu)$ is a skew-symmetric 
endomorphism of $\mathfrak{h}$ with respect to ${g_\phi}_{|\mathfrak{h}}$, for every $\nu\in\mathfrak{s}$. 
In order to compare this example with Theorem \ref{thm:secvar}, we must thus compute 
the tensor $\gamma_Z\in\Lambda^2_-(\Sigma)$, introduced in Lemma \ref{lem:tau2}. One can check that it is given by
\[
\gamma_Z = \left((Z^1)^2+(Z^2)^2+(Z^3)^2\right) \left(e^{45}-e^{67}\right),
\]
for every $Z=Z^1e_1+Z^2e_2+Z^3e_3$.
In particular, one can check that the integrand in the second variation formula vanishes:
\[
(\iota_Zd\tau_2\wedge \iota_Z\phi+\tau_2\wedge\gamma_Z)_{|\Sigma} = 0. 
\]
This corresponds to the fact that all fibres are minimal, so their volume is constant. 
Bryant shows that the fibres are actually calibrated (with respect to an appropriate calibration), thus minimizing. 
We can confirm this by observing that $\iota_Z\tau_2\w d(\iota_Z\phi)_{|\Sigma}=0$. 
Plugging this into the second variation formula yields 
\[
\frac{d^2}{dt^2}\Vol(\Sigma_t)_{|t=0}=\int_\Sigma |\mathcal{C}_Z|^2\iota_0^*\vol_0\geq0,
\]  
for all possible variations. 

\smallskip

\noindent2. Let us now look for a coassociative fibration whose fibres are totally geodesic. 
A standard strategy (see, e.g., \cite{Joy}) is to find a non-trivial isometric involution $\sigma$ of $\mathrm{G}$ such that $\sigma^*\phi=-\phi$: 
each non-trivial connected component of its fixed point set is then a totally geodesic coassociative $4$-fold.

Consider, for example, the restriction to $\mathrm{G}$ of the complex conjugation in $\mathrm{SL}(3,\C)$. 
Let us denote it by $\sigma:\mathrm{G}\to\mathrm{G}$. Its fixed point set is the $4$-dimensional Lie subgroup 
\[
\mathrm{K}\coloneqq \mathrm{Fix}(\sigma) = \left\{
\begin{pmatrix}
\exp(t)	&	z		&	x\\
0		&	\exp(-t)	&	y\\
0		&	0		&	1
\end{pmatrix}
~|~ t,x,y,z\in\R
\right\} \subset \mathrm{G}. 
\]
The Lie algebra $\mathfrak{k}$ of $\mathrm{K}$ is spanned by the left-invariant vector fields $e_1, e_2, e_4,e_7$, and $\mathfrak{g}$ 
admits an $\mathrm{ad}(\mathfrak{k})$-invariant decomposition $\mathfrak{g} = \mathfrak{k} \oplus \mathfrak{m}$, where 
$\mathfrak{m}= \langle e_3,e_5,e_6\rangle$ is an $\mathrm{ad}(\mathfrak{k})$-invariant $3$-dimensional subspace. 
Notice that the differential $\sigma_*:\mathfrak{g}\to\mathfrak{g}$ of $\sigma$ at the identity $1_\mathrm{G}$ satisfies 
$\left.{\sigma_*}\right|_{\mathfrak{k}} = \mathrm{Id}_{\mathfrak{k}}$ and $\left.{\sigma_*}\right|_{\mathfrak{m}} = -\mathrm{Id}_{\mathfrak{m}}$. 
From this, we immediately see that $\sigma^*g_\phi= g_\phi$ and $\sigma^*\phi=-\phi$. 
Therefore, $\mathrm{K}$ is a totally geodesic coassociative $4$-fold of $(\mathrm{G},\phi)$. 
Left translation defines a coassociative fibration $\pi':\mathrm{G}\to\mathrm{G}/\mathrm{K}$ which is not a $g_\phi$-Riemannian submersion, since 
the restriction of $\left.g_\phi\right|_{1_\mathrm{G}}$ to $\mathfrak{m}$ is not $\mathrm{ad}(\mathfrak{k})$-invariant.
Also, the fibres of $\pi'$ cannot be compactified: $\mathrm{K}$ is not unimodular, and thus it does not admit any lattice. 

\smallskip

\noindent3. 
Our final example concerns a coassociative fibration whose fibres are not minimal: this is possible only in the non-calibrated case $d\psi\neq 0$. 

Consider the Abelian ideal $\mathfrak{a} = \langle e_2,e_3,e_6,e_7\rangle$ of $\mathfrak{g}$. Denote by $\mathrm{A}$ the 
unique connected normal Lie subgroup of $\mathrm{G}$ with Lie algebra $\mathfrak{a}$. Since $\phi_{|\mathrm{A}}=0$, 
$\mathrm{A}$ is a coassociative submanifold of $(\mathrm{G},\phi)$. 
Left translation defines a coassociative fibration $\pi'':\mathrm{G}\to\mathrm{G}/\mathrm{A}$. 

As usual, let us compute the mean curvature $H$ of the fibre at the identity of $\mathrm{G}$. 
Since $\mathfrak{a}$ is abelian, we can again use \eqref{IIexB}, 
with $\nu$ belonging to the subalgebra $\mathfrak{a}^\perp = \langle e_1,e_4,e_5 \rangle$ of $\mathfrak{g}$. 
We obtain
\[
\mathrm{II}(e_2,e_2) = e_1,\quad	
\mathrm{II}(e_3,e_3) = e_1,\quad
\mathrm{II}(e_6,e_6) = \frac12\,e_1,\quad	
\mathrm{II}(e_7,e_7) = \frac12\,e_1.
\]
Consequently, 
\[
H = \mathrm{II}(e_2,e_2) + \mathrm{II}(e_3,e_3) + \mathrm{II}(e_6,e_6) + \mathrm{II}(e_7,e_7) = 3\,e_1. 
\] 
Notice that $\tau_{2|\mathrm{A}} = -3\,e^{67} = -\frac32\left(e^{23}+e^{67}\right) -\frac32\left(-e^{23}+e^{67}\right)$. It follows that 
\[
\tau_{2|\mathrm{A}}^+ = -\frac32\left(e^{23}+e^{67}\right) = -\frac12\iota_H\phi_{|\mathrm{A}},
\]
as expected from Section \ref{s:first}.
\end{example}

\begin{example}A second example of ERP manifold is due to Lauret \cite{Lauret}.

Consider the Abelian Lie algebra $\mathfrak{a}$ spanned by all diagonal matrices of $\mathfrak{sl}(4,\R)$, 
and the Abelian Lie algebra $\R^4$. Let $\mathfrak{a}=\langle e_1,e_2,e_3\rangle$, $\R^4=\langle e_4,e_5,e_6,e_7\rangle$, and 
define the seven-dimensional Lie algebra $\mathfrak{g}\coloneqq\mathfrak{a}\ltimes_\mu\R^4$, where  
$\mu:\mathfrak{a}\rightarrow\mathrm{Der}(\R^4)\cong\mathrm{End}(\R^4)$ is given by
\[
\begin{array}{llll}
\mu(e_1)(e_4) = e_4,	& \mu(e_1)(e_5) = e_5,	& \mu(e_1)(e_6) = -e_6,	& \mu(e_1)(e_7) = -e_7,\\
\mu(e_2)(e_4) = e_4,	& \mu(e_2)(e_5) = -e_5,	& \mu(e_2)(e_6) = e_6,	& \mu(e_2)(e_7) = -e_7,\\
\mu(e_3)(e_4) = e_4,	& \mu(e_3)(e_5) = -e_5,	& \mu(e_3)(e_6) = -e_6,	& \mu(e_3)(e_7) = e_7. 
\end{array}
\]
The Lie algebra $\mathfrak{g}$ is solvable and unimodular. 
Its structure equations can be written with respect to the dual basis $(e^1,\ldots,e^7)$ of $(e_1,\ldots,e_7)$ as follows
\[
\left(0, 0, 0,  -e^{14}-e^{24}-e^{34}, -e^{15}+e^{25}+e^{35},  e^{16}-e^{26}+e^{36}, e^{17}+e^{27}-e^{37}   \right).
\]
Let $\mathrm{G}=\mathrm{A}\ltimes\mathbb{R}^4$ be the simply connected solvable Lie group with Lie algebra $\mathfrak{g}$. 
The left-invariant 3-form
\[
\phi = e^{123}+e^{1}\wedge(e^{45}+e^{67})+e^2\wedge(e^{46}-e^{57})-e^3\wedge(e^{47}+e^{56}), 
\]
defines an ERP structure on $\mathrm{G}$ with intrinsic torsion form 
\[
\tau_2 = -2\left(e^{45}-e^{67}\right)-2\left(e^{46}+e^{57}\right)+2\left(e^{47}-e^{56}\right). 
\]

By \cite{KathLauret}, the Lie group $\mathrm{G}$ admits a lattice $\Gamma\subset\mathrm{G}$, and thus one obtains  
a compact $7$-manifold $\Gamma\backslash\mathrm{G}$ endowed with an ERP structure.   
In detail, $\Gamma = \exp(A\Z+B\Z+C\Z) \ltimes \rho(\Z^4)$, where $(A,B,C)$ is a basis of $\mathfrak{a}$ such that $\exp(A)$, $\exp(B)$ and 
$\exp(C)$ leave invariant a lattice $\rho(\Z^4)$ of $\R^4$, for a suitable $\rho\in\mathrm{GL}(4,\R)$.

With the usual identification $\mathfrak{g}\cong T_{1_{\mathrm{G}}}\mathrm{G}$, 
the ideal $\R^4\subset \mathfrak{g}$ is the tangent space at the identity of a coassociative submanifold $\Sigma=\R^4$ 
of $\mathrm{G}$. Since $\tau_{2|\Sigma}\in\Lambda^2_-(\R^4)$, $\Sigma$ is minimal. 
Its second fundamental form at the identity can be computed as in Example \ref{BryEx}. In particular, we obtain 
\[
\begin{array}{ll}
\mathrm{II}(e_4,e_4) = e_1+e_2+e_3,	&	\mathrm{II}(e_5,e_5) = e_1-e_2-e_3,\\
\mathrm{II}(e_6,e_6) = -e_1+e_2-e_3,	&	\mathrm{II}(e_7,e_7) = -e_1-e_2+e_3.
\end{array}
\]
This shows that $\Sigma$ is not totally geodesic. It defines a left-invariant coassociative fibration with flat fibres that compactify to $T^4$. 
We can analyze its properties using exactly the same methods as for Bryant's example.
In particular, it is a Riemannian submersion and the fibres are volume minimizing. 
\end{example}

\begin{remark*}
Further examples of homogeneous spaces admitting invariant ERP structures are known. We refer the reader to \cite{Ball}  
for the complete classification. 
\end{remark*}

\section{Example: Riemannian submersions}\label{s:example}
Let $\Sigma$ be a minimal coassociative submanifold. 
The term $\iota_Zd\tau_{2|\Sigma}\wedge\iota_Z\phi_{|\Sigma}$ in the second variation formula contains a mix of directions, both tangent and orthogonal to $\Sigma$. 
The content of Proposition \ref{prop:Bh} can be seen as providing a reorganization of these contributions, separating the two directions. This reformulation is ideally suited to the context of coassociative fibrations mentioned in Section \ref{s:first}.

We shall assume that the fibration defines a Riemannian submersion $\pi:M\rightarrow B$, so that we have a Riemannian structure on $B$. 
As seen in Section \ref{s:first}, this is a strong condition.

We shall adopt the usual notation $T_pM=T_pM^{\ver}\oplus T_pM^{\hor}$ to describe the splitting defined by $d\pi$. 

Recall that a horizontal vector field $X$ on $M$ is called {\em basic} if, along each fibre, the projections $d\pi(X)$ are constant. 
Whenever an infinitesimal normal deformation $Z$ corresponds to a variation through fibres, it is a basic vector field. 

Let $X,Y$ denote basic vector fields and $V,W$ vertical vector fields. Notice that $\pi_*[X,V]=[\pi_*X,\pi_*V]=[\pi_*X,0]=0$, so $[X,V]=[X,V]^{\ver}$. 

We shall use the notation $e_1,e_2,e_3,e_4$ to denote an orthonormal basis of $T_p\Sigma=T_pM^{\ver}$ and $f_1,f_2,f_3$ to denote a orthonormal 
basis of $T_bB$, where $\pi(p)=b$. Locally, the corresponding basic vector fields define an orthonormal frame of $T\Sigma^\perp=TM^{\hor}_{|\Sigma}$. 

Following O'Neill \cite{ONeill}, set
$$T_VW\coloneqq(\nabla_VW)^{\hor}, \ \ A_XY\coloneqq(\nabla_X Y)^{\ver}.$$
Notice that $X\cdot f_j$ is constant along the fibres, so
\begin{align*}
|(\nabla_XV)^{\hor}|^2&=|(\nabla_VX)^{\hor}|^2=(\nabla_VX\cdot f_j)^2 =(\nabla_V(X\cdot f_j)-X\cdot\nabla_Vf_j)^2\\
&=(X\cdot\nabla_Vf_j)^2 =(X\cdot\nabla_{f_j}V)^2=(\nabla_{f_j}X\cdot V)^2\\
&=(A_{f_j}X\cdot V)^2,
\end{align*}
showing that $A$ controls such terms also. A similar calculation shows that $T$ controls terms of the form $(\nabla_VX)^{\ver}$.

Using these facts, O'Neill provides extensions of $T$, $A$ so that they define tensors on $M$. These two tensors exert strong control over the submersion. 
In particular, $T\equiv 0$ if and only if the fibres are totally geodesic and $A_XY=\frac{1}{2}[X,Y]{^{\ver}}$, 
so $A\equiv 0$ if and only if the horizontal distribution is integrable. 

\begin{ex*}
The Riemannian submersion corresponding to Bryant's ERP manifold, see Example \ref{BryEx}, has $A\equiv 0$.
\end{ex*}

The simplest way to construct a Riemannian submersion is via a free isometric action of a compact Lie group $\mathrm{G}$ on $M$: 
the orbits then define a Riemannian submersion $\pi:M\rightarrow B\coloneqq M/\mathrm{G}$. 
In our context, if we assume that the action preserves $\phi$ then it will also preserve the metric $g_\phi$. 
We shall further assume that the orbits are coassociative submanifolds. 

\begin{remark*}
This construction implies that $\mathrm{G}$ defines automorphisms of the hyper-K\"ahler structure on the fibres, discussed in Section \ref{s:first}. 
K3 surfaces have finite automorphism groups, so they cannot arise as fibres in this construction. The construction thus implies that the fibres are flat tori. 
\end{remark*}

In the above setting we can apply Theorem \ref{thm:secvar} to the fibres. Notice that: 
\begin{enumerate}[(i)]
\item The volume functional on the fibres defines a function $\V:B\rightarrow\R$, $\V(b)\coloneqq\Vol(\Sigma_b)$. 
\item Restricting to coassociative variations corresponds to choosing $Z$ to be a basic vector field on $M$. 
\item The integrand in the second variation formula is invariant under the group action. 
\end{enumerate}

Assume $b\in B$ corresponds to a totally geodesic fibre, so that $T=0$ at each point $p\in \Sigma_b$. 
We can then re-write the second variation formula as follows:
\begin{align*}
\frac{d^2}{dt^2}\V(b_t)_{|t=0}&=\int_{\Sigma_b}\iota_Z d\tau_2\wedge \iota_Z\phi\\
&=\int_{\Sigma_b}(4h(Z,Z)+2\tr(h_{|\Sigma_b})g(Z,Z))\vol\\
&=\left(4h(Z,Z)+2\tr(h_{|\Sigma_b})g(Z,Z)\right)\V(b),
\end{align*}
where $\i(h)=d\tau_2$. Notice: in this context the second variation formula corresponds to the Hessian of $\V$. The above expression shows that it 
depends on the induced bilinear symmetric form $h:{TB \times TB}\rightarrow\R$ and on the induced function $\tr(h_{|\Sigma}):B\rightarrow\R$.

Recall from equation \eqref{eq:dtau2} that $d\tau_2$ contains a term depending on $\star(\tau_2\wedge\tau_2)$ and a term $-\frac{1}{2}\i(\Ric)$. 
One might expect that, in specific situations, the first term can be explicitly calculated: this is the case, for instance, in the example discussed below. 
Here, we are interested in showing how the second term, which appears in the integrand in the form 
$$-\frac12\iota_Z\i(\Ric)\wedge\iota_Z\phi=-2\Ric(Z,Z)-|Z|^2\tr(\Ric_{|\Sigma}),$$
can be related to the curvature of $B$ using O'Neill's curvature formulae.

As usual, let $R$ denote the curvature of $M$. We shall use the convention $R(X,Y)Z\coloneqq\nabla_X\nabla_YZ-\nabla_Y\nabla_XZ-\nabla_{[X,Y]}Z$. 
Let $R^B$ denote the curvature of $B$. \cite{ONeill} (which uses the opposite sign convention for $R$) shows that
\begin{align*}
R(X,Y)Y \cdot X &= R^B(X,Y)Y\cdot X-3|A_XY|^2.
\end{align*}
It follows that
\begin{align*}
\Ric(Z,Z)	&=	R(e_i,Z)Z \cdot e_i+ R(f_j,Z)Z\cdot f_j\\
		&=	R(Z,e_i)e_i\cdot Z+R^B(f_j,Z)Z \cdot f_j -3|A_{f_j}Z|^2 \\
		&=	R(Z,e_i)e_i\cdot Z+\Ric^B(Z,Z) - 3 |A_{f_j}Z|^2.
\end{align*} 
We may also write 
\begin{align*}
\tr(\Ric_{|\Sigma_b})	&=	\Ric(e_i,e_i)
				=	R(e_k,e_i)e_i \cdot e_k + R(f_j,e_i)e_i \cdot f_j 
				=	R(f_j,e_i)e_i \cdot f_j,
\end{align*}
where we use the fact that $\Sigma$ is totally geodesic and Ricci-flat. We conclude that, at $b$,
\begin{align*}
-2\Ric(Z,Z)-|Z|^2\tr(\Ric_{|\Sigma})	&= -2\Ric^B(Z,Z) + 6 |A_{f_j}Z|^2\\
							&\quad -2 R(Z,e_i)e_i \cdot Z - |Z|^2 R(f_j,e_i)e_i \cdot f_j.
\end{align*}
This proves the following result.
\begin{cor}\label{cor:isometric_action}
Let $M$ be endowed with a closed $G_2$-structure $\phi$. 
Assume there exists a group action preserving $\phi$ whose orbits are compact and coassociative. 
Let $\V:B\rightarrow\R$ denote the volume function of the fibres of $\pi:M\to B=M/\mathrm{G}$.  
Assume a fibre $\Sigma_b$ is totally geodesic. Then, at $b$,
\begin{align*}\Hess(\V)(Z,Z)&=(2 k(Z,Z)+|Z|^2\tr(k_{|\Sigma_b})-2\Ric^B(Z,Z))\V(b)\\
&\quad +(6 |A_{f_j}Z|^2-2\,R(Z,e_i)e_i\cdot Z - |Z|^2\,R(f_j,e_i)e_i\cdot f_j)\V(b),
\end{align*}
where $k$ is the unique symmetric $2$-tensor such that  $\i(k)=\star(\tau_2\wedge\tau_2)$.
\end{cor}

\begin{remark*}
The curvature terms above, of the form $R(\hor,\ver)\ver\cdot\hor$, can be further analyzed using O'Neill's formula:
\[
R(X,V)V\cdot X=(\nabla_XT)_VV\cdot X + (\nabla_VA)_XX\cdot V - |T_VX|^2+|A_XV|^2.
\]
\end{remark*}

Recall that Bryant's example does not fit into this framework: the right action of $H$ does not preserve the metric.
A trivial example to which Corollary \ref{cor:isometric_action} does apply is given by the $T^4$-action on the flat $G_2$ manifold $\R^7/\Z^7$. 
In this case $\tau_2=0$. We shall now discuss a non-trivial example due to M.~Fern\'andez \cite{Fernandez}, with non-vanishing torsion.

\begin{example}\label{exFer}
Consider the $2$-step nilpotent matrix group 
\[
\mathrm{N} \coloneqq \left\{
\begin{pmatrix}
1	&	0	&	x^2	&	x^4	&	x^6\\
0	&	1	&	x^3	&	x^5	&	x^7\\
0	&	0	&	1	&	0	&	x^1\\
0	&	0	&	0	&	1	&	0\\
0	&	0	&	0	&	0	&	1
\end{pmatrix}
~|~ x^i\in\R
\right\} \subset \mathrm{GL}(5,\R).  
\]
As a Lie group, $\mathrm{N}\cong \mathrm{H}(1,2)\times\R^2$, where $\mathrm{H}(1,2)$ denotes the $5$-dimensional generalized Heisenberg group. 
A basis of left-invariant $1$-forms on $\mathrm{N}$ is given by 
\[
e^i 	\coloneqq dx^i,~i=1,2,3,4,5,\qquad 
e^6	\coloneqq dx^6-x^2dx^1,\qquad 
e^7	\coloneqq dx^7-x^3dx^1,
\]
so that the structure equations of the Lie algebra $\mathfrak{n}\cong \mathfrak{h}(1,2)\oplus\R^2$ of $\mathrm{N}$ are 
\[
\left(0,0,0,0,0,e^{12},e^{13}\right).
\]
Let $\mathcal{B} = (e_1,\ldots,e_7)$ denote the basis of $\mathfrak{n}$ with dual basis $\mathcal{B}^* = (e^1,\ldots,e^7)$. 
Then, $\mathfrak{h} = \langle e_1,e_2,e_3,e_6,e_7 \rangle$ is a $5$-dimensional ideal of $\mathfrak{n}$ isomorphic to $\mathfrak{h}(1,2)$,  
and we have $\mathfrak{n} = \mathfrak{h}\oplus \mathfrak{a}$, where $\mathfrak{a} = \langle e_4,e_5\rangle$ is a $2$-dimensional Abelian ideal. 

The left-invariant $3$-form on $\mathrm{N}$
\[
\phi = e^{123}+e^{1}\wedge(e^{45}+e^{67})+e^2\wedge(e^{46}-e^{57})-e^3\wedge(e^{47}+e^{56})
\]
defines a closed $G_2$ structure inducing the metric $g_\phi = \sum_{i=1}^7 e^i \otimes e^i$. Its intrinsic torsion form is
\[
\tau_2 = e^{27}-e^{36}. 
\]
Let $\mathrm{Z}$ denote the $4$-dimensional center of $\mathrm{N}$. Its Lie algebra is $\mathfrak{z} = \langle e_4,e_5,e_6,e_7 \rangle$ 
so $\mathrm{Z}$ is a coassociative submanifold of $\mathrm{N}$.

The matrix group $\mathrm{N}$ admits a lattice $\Gamma \coloneqq \mathrm{N} \cap \mathrm{GL}(5,\Z)$, and the left-invariant $3$-form $\phi$ 
descends to an invariant closed $G_2$-structure on the compact quotient $\Gamma \backslash \mathrm{N}$.  
By a result of Eberlein \cite[Prop.~5.5]{Eberlein}, $M\coloneqq \Gamma\backslash\mathrm{N}$ is the total space of a 
Riemannian submersion $\pi:M\to B$ with the following properties:
\begin{enumerate}[$\bullet$]
\item $\mathrm{Z}/(\Gamma\cap\mathrm{Z})$ is a $4$-torus isomorphic to $\mathrm{Iso}^0(M)$;
\item $\mathrm{Iso}^0(M)$ acts freely on $M$ and the orbits are flat, totally geodesic $4$-tori isometric to $\mathfrak{z} / (\log \Gamma \cap \mathfrak{z})$, 
$\log : \mathrm{N}\to\mathfrak{n}$ being the inverse of the Lie group exponential $\exp:\mathfrak{n}\to \mathrm{N}$. 
These orbits are the fibres of $\pi$;
\item the base $B$ is a flat $3$-torus given by the quotient of  $\mathfrak{m}\coloneqq \mathfrak{n}^\perp$, regarded as an additive Abelian group, 
by the lattice $\pi_{\mathfrak{m}}(\log(\Gamma))$, where $\pi_{\mathfrak{m}}:\mathfrak{n}\to\mathfrak{m}$ 
is the orthogonal projection onto $\mathfrak{m}$. 
\end{enumerate}
Since the $G_2$-structure $\phi$ on $\mathrm{N}$ is left-invariant, it follows that the $4$-torus preserves the invariant closed $G_2$-structure on $M,$ 
too. Moreover, the fibres of $\pi$ are coassociative submanifolds of $M.$ 

Since $d\tau_2=2e^{123}$, the restriction of $\iota_Zd\tau_2$ to the fibres vanishes identically for every normal variation $Z$. 
This shows that, if we restrict to variations in the moduli space, then $\Hess(\V)=0$, as expected. 
One can alternatively calculate each term in Corollary \ref{cor:isometric_action}, obtaining the same result. 

\begin{remark*}
If we look at general variations, the facts above and integration by parts allow us to re-write the second variation formula as follows:
$$\frac{d^2}{dt^2}(\Vol(\Sigma_t))_{|t=0}=\int_\Sigma  (\Delta\iota_Z\phi+2d(\iota_Z\tau_2)^+)\w\iota_Z\phi,$$
where all forms are restricted to $\Sigma$, $\Delta$ is the Laplacian on $\Sigma$ and $d(\iota_Z\tau_2)^+$ denotes the self-dual component.
 
This expression can be further modified, taking the form $\int_\Sigma(Q(Z),Z)\vol_0$ or $\int_{\Sigma}(Q(\iota_Z\phi),\iota_Z\phi)\vol_0$, for appropriate operators $Q$. This emphasizes their nature as quadratic forms on, equivalently, the space of normal vector fields or the space of self-dual 2-forms. However, differently from Bryant and Lauret's examples above, it remains unclear whether the fibres are stable. 
\end{remark*}

\begin{remark*}
In examples such as the one above, the vanishing of both tensors $A,T$ would be an exceedingly strong condition.

Indeed, let $M$ be endowed with a closed $G_2$ structure and a map $\pi:M\rightarrow B$ defining a Riemannian submersion. 
If $T\equiv 0$ then \cite{Hermann} shows that the lifting of geodesics from $B$ to $M$ generates isometries between different fibres. 
Restricting to geodesically convex neighbourhoods of $B$, one thus obtains identifications between the fibres. 
If also $A\equiv 0$ we can then build local identifications with the Riemannian product $M\simeq\Sigma\times B$ (see \cite{ONeill} for the global theory).

Assume the fibres are coassociative. 
Choose a local ON frame $e_1,e_2,e_3$ defined on some open subset $U\subseteq B$. The vector fields $e_i$ lift to define infinitesimal 
coassociative deformations of the fibres, so $\omega_i\coloneqq\iota_{e_i}\phi_{|\Sigma}$ are orthogonal harmonic self-dual 2-forms 
on the fibres. 
We may then write $\phi=123+1\wedge \omega_1+2\wedge\omega_2+3\wedge\omega_3$. 
Locally, 123 is a $3$-form on $B$, so $d(123)=0$. Using also $d\omega^i=0$, the condition $d\phi=0$ easily implies that $de_1=de_2=de_3=0$. 
This implies that the vector fields $e_1,e_2,e_3$ commute, so they define local ON coordinates on $B$. 
It then follows that $M$ is locally of the form $\Sigma\times \R^3$, so $\phi$ is torsion-free. 

$M$ is then Ricci-flat, so many terms in O'Neill's curvature formulae vanish.
\end{remark*}

\end{example}

\section{Example: perturbations}\label{s:perturbations}

The previous examples provide a good testing ground to verify our formulae, but their homogeneous structure implies that the fibres have constant volume. 
The first and second variations thus vanish automatically.

Homogeneity simplifies the construction of examples, but is not required by Theorem \ref{thm:secvar}. 
The theorem does however simplify when the initial submanifold is totally geodesic. 
We will now show how to produce an infinite set of examples of this type via perturbation of the homogeneous example seen in Section \ref{s:example}. 
The construction relies on an idea already used in Section \ref{s:erp}.

\smallskip

Let $\mathrm{N}$ be as in Example \ref{exFer}. Consider the map 
\[
c:\mathrm{N}\rightarrow\mathrm{N}, \qquad
\begin{pmatrix}
1	&	0	&	x^2	&	x^4	&	x^6\\
0	&	1	&	x^3	&	x^5	&	x^7\\
0	&	0	&	1	&	0	&	x^1\\
0	&	0	&	0	&	1	&	0\\
0	&	0	&	0	&	0	&	1
\end{pmatrix}\mapsto
\begin{pmatrix}
1	&	0	&	-x^2	&	x^4	&	x^6\\
0	&	1	&	-x^3	&	x^5	&	x^7\\
0	&	0	&	1	&	0	&	-x^1\\
0	&	0	&	0	&	1	&	0\\
0	&	0	&	0	&	0	&	1
\end{pmatrix}.
\]
One can easily check that $c$ is a group homomorphism preserving the lattice $\Gamma$. 
It thus induces a map on $M\coloneqq \Gamma\backslash\mathrm{N}$ which we will continue to denote by $c$. 
It is also simple to check that $c^*\phi=-\phi$ and that $c^2=\mathrm{Id}$. 
The $T^4$-fibre $\Sigma$ passing through $[1_{\mathrm{N}}]\in M$ coincides with the fixed point set of $c$. 
As in Example \ref{BryEx}, this confirms that $\Sigma$ is totally geodesic and coassociative.

Now choose any 2-form $\alpha$ on $M$. Set $\alpha' \coloneqq c^*\alpha-\alpha$ and $\phi' \coloneqq \phi+\epsilon d\alpha'$, for $\epsilon\in\R$. 
Then:
\begin{enumerate}[1.]
\item The condition of being a $G_2$ structure is open, so $\phi'$ is a closed $G_2$ structure for any sufficiently small $\epsilon$.
\item Clearly $c^*\phi'=-\phi'$, so the usual argument proves that the same fixed point set $\Sigma$ is coassociative also for $\phi'$ 
and totally geodesic for its induced metric.
\item The topology of $\Sigma$ has not changed, so the moduli space $\mathcal{M}'$ of $\phi'$-coassociative deformations of $\Sigma$ 
is again $3$-dimensional.
\end{enumerate}
The generic such $\alpha$ leads to a metric for which $\Sigma$ has non-trivial second variations. 

At this level of generality it is however unclear whether the new moduli space $\mathcal{M}'$ corresponds to a new coassociative fibration of $M$: 
the deformed submanifolds might now intersect each other. We can obtain further properties by restricting the class of 2-forms, as follows.

The fact that $c:M\rightarrow M$ maps fibres into fibres implies that it descends to a map $c:B\rightarrow B$. 
The two maps are related by the property $\pi\circ c=c\circ\pi$.

If we choose $\alpha$ to be a 2-form on $B$, we can set $\alpha' \coloneqq c^*\alpha-\alpha$ on $B$, 
then consider the $3$-form $\phi' \coloneqq \phi+\epsilon d(\pi^*(\alpha'))$ on $M$. 
Once again $c^*\phi'=-\phi'$, but now:
\begin{enumerate}[1.]
\item[4.]	
The perturbation term $d(\pi^*(\alpha'))$ is of the form $\lambda\, e^{123}$, for some $\lambda = \lambda(x^1,x^2,x^3).$ 
If $(e^1,\ldots,e^7)$ denotes the invariant coframe on $M$ induced by $\mathcal{B}^*$, we have
\[
\phi' = \left(1+\epsilon\lambda\right)e^{123}+e^{1}\wedge(e^{45}+e^{67})+e^2\wedge(e^{46}-e^{57})-e^3\wedge(e^{47}+e^{56}).
\] 
Linear algebra shows that 
\[
g_{\phi'} = \left(1+\epsilon\lambda\right)^{2/3}\sum_{i=1}^3 e^i\otimes e^i + \left(1+\epsilon\lambda\right)^{-1/3}\sum_{i=4}^7 e^i\otimes e^i. 
\]
Thus, $\phi$ and $\phi'$ induce homothetic metrics on any fibre.

\item[5.] Consider the $T^4$-action on $M$ described in Example \ref{exFer}. 
Given $g\in T^4$, $\pi\circ g=\pi$ so $g^*\pi^*=\pi^*$. It follows that $g^*(\pi^*(\alpha'))=\pi^*(\alpha')$. 
Since also $g^*\phi=\phi$, it follows that $g^*(\phi')=\phi'$. 
In particular, the action on $M$ is again isometric with respect to the new metric induced by $\phi'$.
\item[6.] Both $\phi$ and $d(\pi^*(\alpha'))$ vanish tangentially to the $T^4$-orbits, so these are coassociative with respect to $\phi'$. 
\end{enumerate} 

We conclude that the same fibration of $M$, already discussed in Example \ref{exFer}, is again a Riemannian submersion with respect to the new metric, 
and that its fibres are again flat and coassociative. As above, $\Sigma$ is totally geodesic so its first variations vanish, while its second variations are generally non-trivial.

\bibliographystyle{amsplain}
\bibliography{coassvolume_biblio}

\end{document}